\documentclass[reqno, 11pt]{amsart}
\usepackage[percent]{overpic}
\usepackage{calc,graphicx,amsfonts,amsthm,amscd,amsmath,amssymb,enumerate,dsfont,mathrsfs,paralist}
\usepackage{subfig}
\usepackage{pdfsync}
\usepackage{stmaryrd}
\usepackage[numeric,initials,nobysame]{amsrefs}
\usepackage{marginnote}
\usepackage[top=2.5cm, bottom=2.5cm, outer=2.5cm, inner=2.5cm, heightrounded, marginparwidth=1.9cm, marginparsep=0.3cm]{geometry}

\usepackage{color}
\usepackage{booktabs}

\newcommand{\ie}{\hbox{\it i.e.\ }}

\definecolor{light}{gray}{.9}

\reversemarginpar
\newlength\fullwidth
\numberwithin{equation}{section}

\DeclareMathSymbol{\leqslant}{\mathalpha}{AMSa}{"36} % nicer `smaller or equal'
\DeclareMathSymbol{\geqslant}{\mathalpha}{AMSa}{"3E} % nicer `larger or equal'
\DeclareMathSymbol{\eset}{\mathalpha}{AMSb}{"3F}     % nicer `emptyset'
\renewcommand{\leq}{\;\leqslant\;}                   % redef. of < or =
\renewcommand{\geq}{\;\geqslant\;}                   % redef. of > or =
\newcommand{\sumtwo}[2]{\sum_{\substack{#1 \\ #2}}} % sum with 2 lines
\renewcommand{\b}{\beta}
\newcommand{\TV}{{\textsc{tv}}}
\newcommand{\eps}{\epsilon}

\def\1{\ifmmode {1\hskip -3pt \rm{I}} \else {\hbox {$1\hskip -3pt \rm{I}$}}\fi}

\newcommand{\var}{\operatorname{Var}}
\newcommand{\cov}{\operatorname{Cov}}

\newcommand{\tmix}{T_{\rm mix}}
\newcommand{\trel}{T_{\rm rel}}

\newcommand{\D}{\Delta}

\renewcommand{\b}{\beta}
\renewcommand{\l}{\lambda}
\renewcommand{\L}{\Lambda}

\renewcommand{\l}{\lambda}
\renewcommand{\a}{\alpha}
\renewcommand{\d}{\delta}
\renewcommand{\t}{\tau}

\newcommand{\g}{\gamma}

\newcommand{\e}{\varepsilon}

\renewcommand{\o}{\omega}
\renewcommand{\O}{\Omega}
\newcommand{\East}{{\rm E}}
\newcommand{\Shift}{{\rm S}}
\newcommand{\Front}{{\rm F}}
\newcommand{\gap}{{\rm gap}}
\newcommand{\tc}{\thinspace |\thinspace}

\newtheorem{theorem}{Theorem}[section]
\newtheorem{lemma}[theorem]{Lemma}
\newtheorem{proposition}[theorem]{Proposition}
\newtheorem{corollary}[theorem]{Corollary}
\newtheorem{remark}[theorem]{Remark}

\newtheorem{claim}[theorem]{Claim}
\newtheorem{definition}[theorem]{Definition}
\newtheorem{maintheorem}{Theorem}

\newtheorem*{question*}{Question}

\newtheorem*{remark*}{Remark}
\newtheorem*{idefinition*}{Definition}

\newcommand{\N}{\mathbb N}

\newcommand{\cB}{\ensuremath{\mathcal B}}
\newcommand{\cC}{\ensuremath{\mathcal C}}

\newcommand{\cF}{\ensuremath{\mathcal F}}
\newcommand{\cG}{\ensuremath{\mathcal G}}

\newcommand{\cL}{\ensuremath{\mathcal L}}

\newcommand{\cN}{\ensuremath{\mathcal N}}

\newcommand{\cS}{\ensuremath{\mathcal S}}
\newcommand{\cT}{\ensuremath{\mathcal T}}

\newcommand{\cW}{\ensuremath{\mathcal W}}

\newcommand{\bbE}{{\ensuremath{\mathbb E}} }

\newcommand{\bbN}{{\ensuremath{\mathbb N}} }

\newcommand{\bbP}{{\ensuremath{\mathbb P}} }

\newcommand{\bbR}{{\ensuremath{\mathbb R}} }

\newcommand{\bbT}{{\ensuremath{\mathbb T}} }

\newcommand{\bbZ}{{\ensuremath{\mathbb Z}} }
\newcommand{\Z}{{\ensuremath{\mathbb Z}} }

\newcommand{\E}{\mathbb{E}}

\let\a=\alpha \let\b=\beta   \let\d=\delta  \let\e=\varepsilon
 \let\g=\gamma       \let\l=\lambda
\let\m=\mu   \let\n=\nu   \let\o=\omega    \let\p=\pi  
  \let\s=\sigma \let\t=\tau   
  
\let\D=\Delta     \let\L=\Lambda 
\let\O=\Omega      

\renewcommand{\le}{\leq}

\title{Cutoff for the East process}

\begin{document}

\author[S. Ganguly]{S.\ Ganguly}
 \address{Shirshendu Ganguly, \hfill\break
 Department of Mathematics\\ University of Washington\\
%C-541 Padelford Hall\\
Seattle, WA 98195, USA.}
\email{sganguly@math.washington.edu}

\author[E. Lubetzky]{E.\ Lubetzky}
%\address{E.\ Lubetzky\hfill\break
%Microsoft Research\\ One Microsoft Way\\ Redmond, WA 98052, USA.}
%\email{eyal@microsoft.com}
\address{Eyal Lubetzky\hfill\break
Courant Institute %of Mathematical Sciences
\\ New York University\\
251 Mercer Street\\ New York, NY 10012, USA.}
\email{eyal@courant.nyu.edu}

 \author[F. Martinelli]{F.\ Martinelli}
 \address{Fabio Martinelli\hfill\break
 Dip.\ Matematica \& Fisica\\
   Universit\`a Roma Tre\\ Largo S.\ Murialdo 1\\ 00146 Roma, Italy.}
\email{martin@mat.uniroma3.it}

\begin{abstract}
The East process is a 1\textsc{d} kinetically constrained interacting particle system,
introduced in the physics literature in the early 90's to model liquid-glass transitions.
Spectral gap estimates of Aldous and Diaconis in 2002
imply that its mixing time on
$L$ sites has order $L$. We complement that result and show cutoff with an $O(\sqrt{L})$-window.

The main ingredient is an analysis of the \emph{front} of the process (its rightmost zero in the setup where zeros facilitate updates to their right).
One expects the front to advance as a biased random walk, whose normal fluctuations would imply cutoff with an $O(\sqrt{L})$-window.
The law of the process behind the front plays a crucial role:
Blondel showed that it converges to an invariant measure $\nu$, on which very little is known. Here we obtain quantitative bounds on the speed of convergence to $\nu$, finding that it is exponentially fast. We then  derive that the increments of the front behave as a stationary mixing sequence of random variables, and a Stein-method based argument of Bolthausen (`82) implies a CLT for the location of the front, yielding the cutoff result.

Finally, we supplement these results by a study of analogous kinetically constrained models on trees, again establishing cutoff, yet this time with an $O(1)$-window.
\end{abstract}
{\baselineskip18\p\
\maketitle
}
\vspace{-0.5cm}
%\tableofcontents[1]

\section{Introduction}

The East process is a one-dimensional spin system that was introduced in the physics literature by J\"{a}ckle and Eisinger~\cite{JE91} in 1991
to model the behavior of cooled liquids near the glass transition point, specializing a class of models that goes back to~\cite{FH}.
 Each site in $\Z$ has a $\{0,1\}$-value (vacant/occupied), and, denoting this configuration by $\omega$, the process attempts to update $\omega_x$ to $1$ at rate $0<p<1$ (a parameter)
and to $0$ at rate $q=1-p$, only accepting the proposed update if
$\omega_{x-1}=0$ (a ``kinetic constraint''). 

It is the properties of the East process before and towards reaching equilibrium --- it is reversible w.r.t.\ $\pi$, the product of Bernoulli($p$) variables ---
which are of interest, with the standard gauges for the speed of convergence to stationarity being the inverse spectral-gap and the total-variation mixing time ($\gap^{-1}$ and $\tmix$)
on a finite interval $\{0,\ldots,L\}$, where we fix $\omega_0=0$ for ergodicity (postponing formal definitions to \S\ref{sec:prelims}).
That the spectral-gap is uniformly bounded away from 0 for any
$p\in(0,1)$ was first proved in a beautiful work of Aldous and
Diaconis~\cite{AD02} in 2002.
This implies that $\tmix$
is of order $L$ for any fixed threshold $0<\epsilon<1$ for the total-variation distance from $\pi$.

For a configuration $\omega$ with $\sup\{x:\omega_x=0\}<\infty$, call this rightmost 0 its \emph{front} $X(\omega)$; key questions on the East process $\omega(t)$ revolve the law $\mu^t$ of the sites behind the front at time $t$,
basic properties of which remain unknown. One can imagine that the front advances to the right as a biased walk, behind which $\mu^t \approx \pi$ (its trail is mixed).
 Indeed, if one (incorrectly!) ignores dependencies between sites as well as the randomness in the position of the front, it is tempting to conclude that $\mu^t$ converges to $\pi$, since upon updating a site $x$ its marginal is forever set to Bernoulli($p$). Whence, the positive vs.\ negative increments to $X(\omega)$ would have rates $q$ (a 0-update at
$X(\omega)+1$) vs.\ $pq$ (a 1-update at $X(\omega)$ with a 0 at its
left), giving the front an asymptotic speed $v =q^2>0$.

Of course, ignoring the irregularity near the front is problematic, since it is precisely the distribution of those spins that governs the speed of the front (hence mixing).
Still, just as a biased random walk, one expects the front to move at a positive speed with normal fluctuations,
whence its concentrated passage time through an interval would imply total-variation \emph{cutoff} --- a sharp transition in mixing --- within an $O(\sqrt{L})$-window.

To discuss the behavior behind the front, let $\Omega_\Front$ denote the set of configurations $\omega^\Front$ on the negative half-line $\Z_-$ with a fixed 0 at the origin, and let $\omega^\Front(t)$ evolve via the East process constantly re-centered (shifted by at most 1) to keep its front at the origin.
Blondel~\cite{Blondel} showed  (see Theorem~\ref{Blondel1}) that the process $\omega^\Front(t)$
converges to an invariant measure $\nu$, on which very little is
known, and that $\frac1{t} X(\omega(t)) $ converges in probability
to a positive limiting value $v$ as $t\to\infty$ (an asymptotic velocity)
given by the formula
\[
v = q-pq^* \qquad\mbox{ where }\qquad q^* := \nu(\omega_{-1}=0) .
\]
(We note that $q < q^* < q/p$ by the invariance of the measure $\nu$ and the fact that $v>0$.)

The East process $\omega(t)$ of course entails the joint distribution of $\omega^\Front(t)$ and $X(\omega(t))$; thus, it is crucial to understand the dependencies between these as well as the rate at which $\omega^\Front(t)$ converges to $\nu$
as a prerequisite for results on the fluctuations of $X(\omega(t))$.

Our first result confirms the biased random walk intuition for the front of the East process $X(\omega(t))$,
establishing a CLT for its fluctuations around $v t$ (illustrated in Fig.~\ref{fig:front}).

\begin{figure}[t]
%\centering
\includegraphics[width=.75\textwidth]{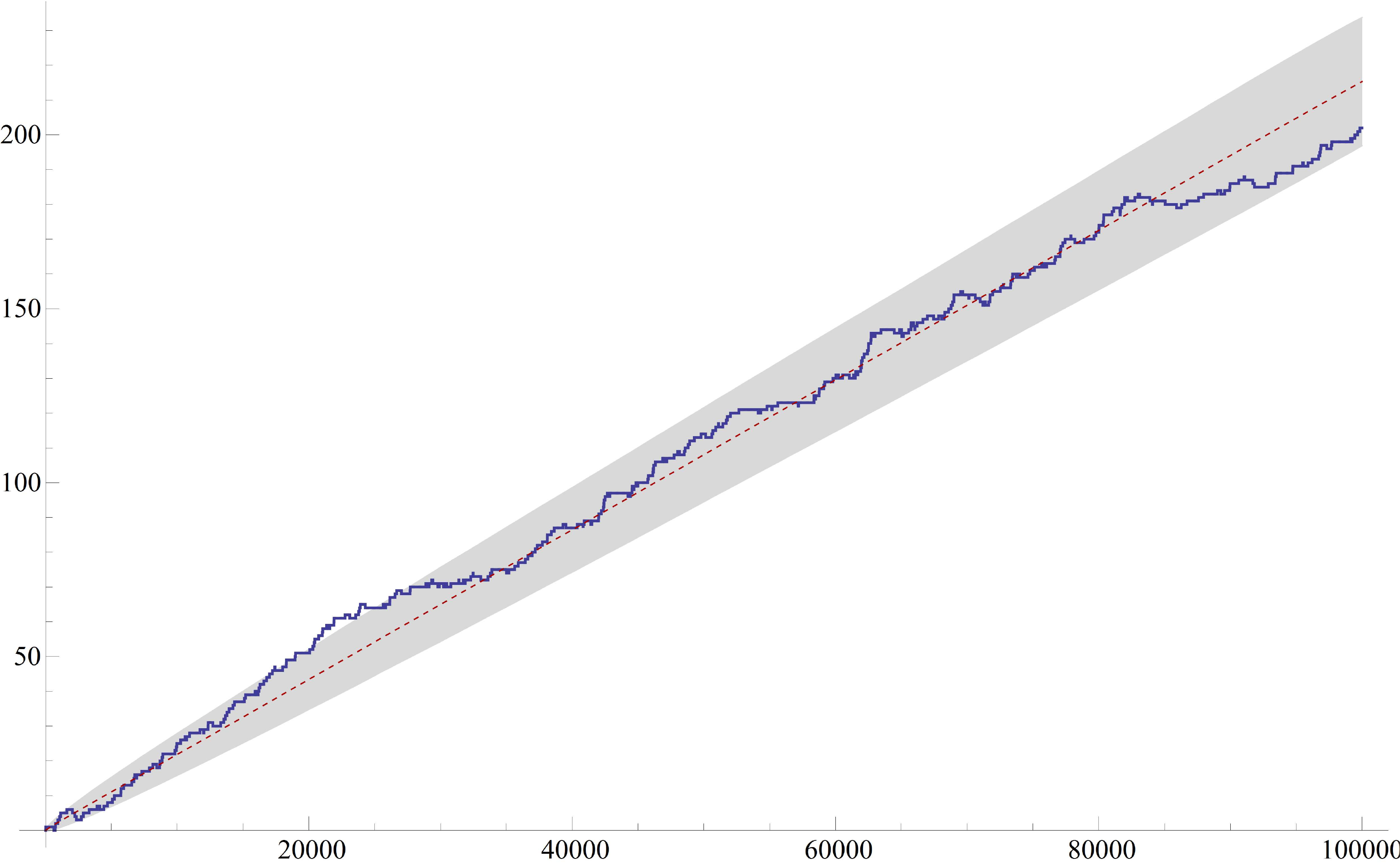}
\caption{Trajectory of the front of an East process for $p=\frac14$ along a time interval of $10^5$, vs.\ its mean and standard deviation window.}
\label{fig:front}
%\vspace{-0.25cm}
\end{figure}

\begin{maintheorem}\label{th:main1} There exists a non-negative constant
  $\s_*=\s_*(p)$ such that for all $\o\in \O_\Front$,
\begin{align}
  \label{th1.1}
\lim_{t\to \infty} \tfrac 1t X(\o(t))&=v  \quad \bbP_\o\text{-a.s},\\
\label{th1.2}
\bbE_\o\left[X(\o(t))\right]&=vt+O(1),\\
\label{th1.3}
\lim_{t\to \infty} \tfrac 1t \var_\o\left(X(\o(t))\right)&=\s_*^2.
\end{align}
Moreover, $X(\o(t))$ obeys a
central limit theorem:
\begin{equation}
  \label{th1.4}
\frac{X(\o(t))-v
    t}{\sqrt{t}} \overset{d}{\rightarrow} \cN(0,\s_*^2) \quad w.r.t.\quad \bbP_\o\mbox{ as $t\to\infty$}.
\end{equation}
\end{maintheorem}

A key ingredient for the proof is a quantitative bound on the rate of convergence to $\nu$, showing that it is exponentially fast (Theorem~\ref{coupling}). We then show that the increments
\begin{equation}
  \label{eq-xi_n-def}
\xi_n:=X(\o(n))-X(\o(n-1)) \qquad  (n\in \bbN)
\end{equation}
behave (after an initial burn-in time) as a stationary sequence of weakly dependent random variables (Corollary~\ref{cor:wf}), 
whence one can apply an ingenious Stein's-method based argument of Bolthausen~\cite{Bolthausen} from 1982 to derive the CLT.

Moving our attention to finite volume, recall that
the \emph{cutoff phenomenon} (coined by Aldous and Diaconis~\cite{AD86}; see~\cites{Aldous,DiSh} as well as~\cite{Diaconis} and the references therein) describes a sharp transition in the convergence of a finite Markov chain to stationarity: over a negligible period of time (the cutoff window) the distance from equilibrium drops from near 1 to near $0$. Formally, a sequence of chains indexed by $L$ has cutoff around $t_L$ with window $w_L=o(t_L)$ if $\tmix(L,\epsilon) = t_L + O_\epsilon(w_L)$ for any fixed $0<\epsilon<1$.

It is well-known (see, e.g.,~\cite{DiFi}*{Example 4.46}) that a biased random walk with speed $v>0$ on an interval of length $L$ has cutoff at $v^{-1}L$ with an $O(\sqrt{L})$-window due to normal fluctuations. Recalling the heuristics that depicts the front of the East process as a biased walk flushing a law $\mu^t\approx\pi$ in its trail, one expects precisely the same cutoff behavior. Indeed, the CLT in Theorem~\ref{th:main1} supports a result exactly of this form.

\begin{maintheorem}
\label{th:main2}
The East process on $\L=\{1,2,\dots ,L\}$ with parameter $0<p<1$
exhibits cutoff at $v^{-1}L$ with an $O(\sqrt{L})$-window:
for any fixed $0<\epsilon<1$ and large enough $L$,
\begin{align*}
 \tmix(L,\eps)&= v^{-1} L + O\left(\Phi^{-1}(1-\eps)\, \sqrt{L}\right),
\end{align*}
where $\Phi$ is the c.d.f.\ of $\cN(0,1)$ and the implicit constant in the $O(\cdot)$ depends only on $p$.
\end{maintheorem}

\begin{figure}[t]
%\centering
\vspace{-0.5cm}
\includegraphics
[width=.75\textwidth]{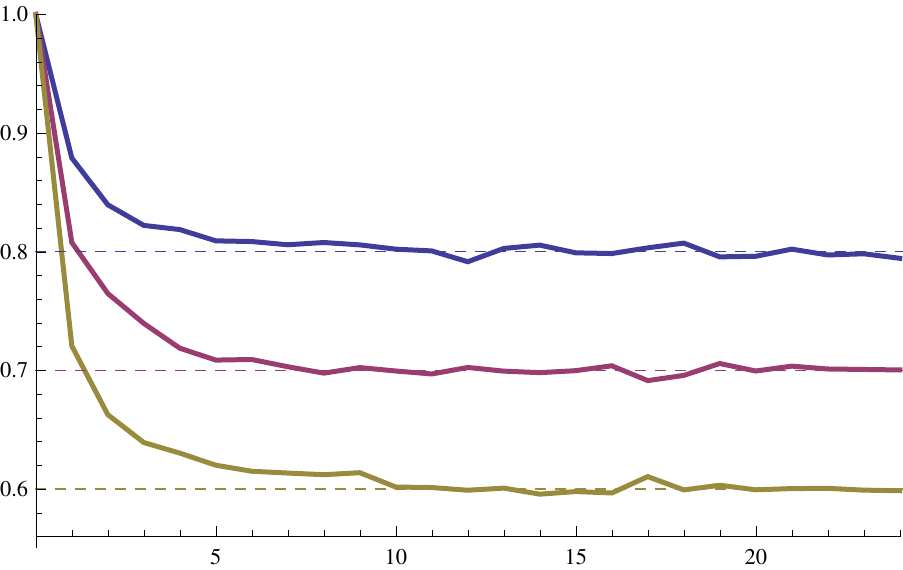}
\vspace{-0.35cm}
\caption{The invariant measure $\nu$ behind the front of the East process
(showing $\nu(\o_{-i}=0)$ simulated
via Monte-Carlo for $p\in\{0.2,0.3,0.4\}$.)}
\label{fig:nu}
\vspace{-0.5cm}
\end{figure}

While these new results relied on a refined understanding of the
convergence of the process behind the front to its invariant law
$\nu$ (shown in Fig.~\ref{fig:nu}), various basic
questions on $\nu$ remain unanswered. For instance, are the single-site
marginals of
$\nu$ monotone in the distance from the front? What are the correlations between adjacent spins? Can one explicitly obtain $q^*=\nu(\omega_{-1}=0)$, thus yielding an expression for the velocity $v$?
For the latter, we remark that the well-known upper bound on $\tmix$ in terms of the spectral-gap (Eq.~\eqref{eq-tmix-gap}), together with Theorem~\ref{th:main2}, gives the lower bound (cf.\ also~\cite{CFM})
\[
v\ge \limsup_{L\to\infty}\frac{\gap(\cL_{[0,L]})}{\log\left(1/(p\wedge q)\right)}=\frac{\gap(\cL)}{\log\left(1/(p\wedge q)\right)}.
\]
\smallskip

Finally, we accompany the concentration for $X(\omega(t))$ and cutoff for the East process by analogous results --- including cutoff with an $O(1)$-window --- on the corresponding kinetically constrained models on trees, where a site is allowed to update (i.e., to be reset into a Bernoulli($p$) variable) given a certain configuration of its children (e.g., all-zeros/at least one zero/etc.). These results are
detailed in \S\ref{sec:trees} (Theorems~\ref{th:main3}--\ref{th:main4}).

\begin{remark*}
The concentration and cutoff results for the kinetically constrained models on trees (Theorems~\ref{th:main3}--\ref{th:main4})
do not apply to every scale but rather to infinitely many scales, as is sometimes the case in the context of tightness for maxima of branching random walks or discrete Gaussian Free Fields;
see, e.g.,~\cites{BDZ,DH91} as well as the beautiful method in~\cites{BZ1,BZ2} to overcome this hurdle for certain branching random walks. Indeed, similarly to the latter, one of the models here gives rise to a distributional recursion involving the maximum of i.i.d.\ copies of the random variable of interest, plus a non-negative increment. Unfortunately, unlike branching random walks, here this increment is not independent of those two copies, and extending our analysis to every scale appears to be quite challenging.
\end{remark*}

\section{Preliminaries and tools for the East process}\label{sec:prelims}
\subsection{Setup and notation}
\label{setting-notation}
Let $\O=\{0,1\}^\bbZ$ and let $
\O^*\subset \O$ consist of those configurations $\o\in \O$ such that the
variable $X(\o):=
\sup\{x:\omega_x=0\}$ is finite.
In the sequel, for any $\o\in \O^*$  we will often refer to $X(\o)$ as
the \emph{front} of $\o$. Given $\L\subset \bbZ$ and $\o\in \O$ we
will write $\o_\L$ for the restriction of $\o$ to $\L$.

\begin{enumerate}[(i)]
\item \emph{The East process.} For any $\o\in \O$ and $x\in \bbZ$ let $c_x(\o)$ denote the indicator
of the event $\{\o_{x-1}=0\}$. We will consider the Markov
process $\{\o(t)\}_{t\ge 0}$ on $\O$ with generator acting on local functions (\ie depending
on finitely many coordinates) $f:\O\mapsto \bbR$   given by
\[
\cL f(\o)=\sum_{x\in \bbZ}c_x(\o)\left[\pi_x(f)(\o)-f(\o)\right],
\]
where $\pi_x(f)(\o):= pf(\o^{(x,1)})+qf(\o^{(x,0)})$
and $\o^{(x,1)},\o^{(x,0)}$ are the configurations in $\O$ obtained from $\o$ by fixing equal to $1$
or to $0$ respectively the coordinate at $x$. In the sequel the above
process will be referred to as the \emph{East
  process on $\bbZ$} and we will write $\bbP_\o(\cdot)$ for its law
when the starting configuration is $\o$. Average and variance w.r.t.\ to
$\bbP_\o(\cdot)$ will be denoted by $\bbE_\o[\cdot]$ and
$\var_\o(\cdot)$ respectively. Similarly we
will write $\bbP_\o^t(\cdot)$ and $\bbE_\o^t[\cdot]$ for the law and
average at a fixed time $t>0$. If the starting
configuration is distributed according to an initial distribution
$\eta$ we will simply write  $\bbP_\eta(\cdot)$ for $\int
d\eta(\o)\bbP_\o(\cdot)$ and similarly for $\bbE_\eta[\cdot]$.

It is easily seen that the East process has the following graphical
representation. To
each $x\in\bbZ$ we associate a rate-1 Poisson process and,
independently, a family of independent Bernoulli$(p)$ random variables
$\{s_{x,k} : k \in \N\}$. The occurrences of the Poisson process
associated to $x$ will be denoted by $\{t_{x,k} : k \in \N\}$.  We
assume independence as $x$ varies in $\bbZ$. That fixes the
probability space. Notice that almost surely all
the occurrences $\{t_{x,k}\}_{k\in\N, x\in\bbZ}$ are different.
On the above probability we construct a Markov process according to the following rules. At each time
$t_{x,n}$ the site $x$ queries the state of its own constraint $c_x$.
If and only if the constraint is satisfied ($c_x = 1$) then $t_{x,n}$
is called a \emph{legal ring} and the configuration resets its
value at site $x$ to the value of the corresponding Bernoulli variable
$s_{x,n}$.  Using the graphical construction it is simple to see that if $
\o\in \O^*$
then
\[
\bbP_\o(\o(t)\in \O^* \, \forall t\ge 0)=1.
\]
\item \emph{The half-line East process.} Consider now $a\in \bbZ$ and let $\O^a$ consist of those
configurations $\o\in \O$ with a \emph{leftmost} zero at $a$.
Clearly, for any $\o\in \O^a$,
$\bbP_\o(\o(t)\in \O^a\ \forall t>0)=1$
because $c_x(\o)=0$ for any $x\le a$. We will refer to the corresponding process in $\O^a$
as the East process on the half-line $(a,\infty)$. Notice that in this case the
variable at $a+1$ will always be unconstrained because $c_a(\o)=1$ for
all $\o\in \O^a$.  The corresponding generator will be denoted by $\cL_{(a,\infty)}$.
\item \emph{The finite volume East process.} Finally, if $\L\subset \bbZ$ is
a discrete interval of the form
$\L=[a+1,\dots a+L]$, the projection on $\O_\L\equiv \{0,1\}^\L$ of the half-line East process on
$(a,\infty)$ is a continuous time Markov chain because each vertex
$x\in \L$ only
queries the state of the spin to its left. In the sequel the above chain will be referred to as the
\emph{East process} in $\L$. Let $\cL_{\L}$ denote the corresponding generator.
\end{enumerate}

The main properties of the above processes can be summarized as
follows (cf.~\cite{East-survey} for a survey). They are all ergodic and
reversible w.r.t.\ to the product Bernoulli($p$) measure $\pi$ (on the
corresponding state space). Their generators
$\cL,\cL_{(a,\infty)},\cL_\L$ are self-adjoint operators on
$L^2(\pi)$ satisfying the following
natural ordering:
\[
\gap(\cL)\le \gap(\cL_{(a,\infty)})\le \gap(\cL_{\L}).
\]
\begin{remark*}
By translation invariance the value of $\gap(\cL_{(a,\infty)})$ does not
depend on $a$ and, similarly, $\gap(\cL_\L)$ depends only on the cardinality of $\L$.
\end{remark*}
As mentioned before, the fact that $\gap(\cL)>0$ (but only for $p\sim 1$) was first proved by Aldous and
Diaconis~\cite{AD02}, where it was further shown that
\begin{align}
  \label{eq-AD-gap}
  e^{-(\frac{1}{\log 2}+o(1))\log^2(1/q)} \leq \gap(\cL) \leq e^{-(\frac{1}{2\log 2}+o(1))\log^2(1/q)} \quad \mbox{as $q\downarrow 0$},
\end{align}
the order of the exponent in the lower bound matching non-rigorous predictions in the physics literature.
The positivity  of $\gap(\cL)$ was rederived and extended to all $p\in (0,1)$
in~\cite{CMRT} by different methods, and the correct asymptotics
of the exponent as
$q\downarrow 0$ --- matching the \emph{upper bound} in~\eqref{eq-AD-gap} ---
was very recently established in~\cite{CFM}. It is easy to check (e.g., from~\cite{CMRT})
that $\lim_{p\to 0}\gap(\cL)=1$, a fact that
will be used later on.

For the East process in $\L$ it is natural to consider its mixing times
$\tmix(L,\eps)$, $\eps\in (0,1)$, defined by
\[
\tmix(L,\eps)=\inf\Big\{t:\ \max_{\o\in \O_\L}\|P_\o^t(\cdot)-\pi\|\le \eps\Big\},
\]
where $\|\cdot \|$ denotes
total-variation distance.
It is a standard result for reversible Markov chains (see
e.g.~\cites{AF,LPW,Saloff}) that
\begin{equation}
  \label{eq-tmix-gap}
\tmix(L,\eps) \le \frac 12 \gap(\cL_\L)^{-1}\left(2+\log
  \frac{1}{\pi_\L^*}\right) \log \frac1\eps ,
\end{equation}
where $\pi_\L^*:= \min_{\o\in \O_\L}\pi(\o)$. In particular
$\tmix(L,\eps)\le c(p) L \log 1/\eps$.
A lower bound which also grows linearly in the length $L$ of the
interval $\L$ follows easily from the \emph{finite speed of
  information propagation}:
   If we run the East
model in $\L$ starting from the configuration of $\omega\equiv 1$
except for a zero at the origin, then, in order to create zeros
near the right boundary of $\L$ a sequence of order $L$ of successive
rings of the Poisson clocks at consecutive sites must have
occurred. That happens with probability $O(1)$ iff we allow a time
which is linear in $L$ (see \S\ref{sec:finite-speed} and in particular Lemma~\ref{finitespeed}).

\subsection{The process behind the front}\label{sec:proc-front}

Given two probability measures $\nu,\mu$ on $\O$ and $\L\subset
\bbZ$ we will write $\|\mu-\nu\|_\L$ to denote the total variation
distance between the marginals of $\mu$ and $\nu$ on $\O_\L=\{0,1\}^\L$.

When the process starts from a initial configuration $\o\in \O^*$ with
a front, it is convenient to
define a new process $\{\o^\Front(t)\}_{t\ge 0}$ on $\O_\Front:=\{\o\in \O^*:\ X(\o)=0\}$ as \emph{the
  process as seen from the front} \cite{Blondel}.  Such a process is obtained
from the original one by a random shift $-X(\o(t))$
which forces the front to be always at the origin. More precisely
we define on $\O_\Front$ the Markov process with generator $\cL^\Front=\cL^\East+\cL^\Shift$ given by
\begin{align*}
%\label{proc-front}
\cL^\East f(\o) &=
\sum_{x<0}c_x(\o)\left[\pi_x(f)(\o)-f(\o)\right], \\
\cL^\Shift f(\o)&=
(1-p)\left[f(\vartheta^-\o)-f(\o)\right]+ p \,c_0(\o)\left[f(\vartheta^+\o)-f(\o)\right],
\end{align*}
where
\begin{equation*}
\left( \vartheta^\pm\o\right)_x=
\begin{cases}
0&\text{ if $x=0$}\\
1 &\text{ if $x>0$}\\
\o_{x\mp1}&\text{ otherwise}.
\end{cases}
\end{equation*}
That is, the generator $\cL^\Front$ incorporates the moves of the East
process behind the front plus $\pm1$ shifts corresponding to whenever the front itself jumps forward/backward.
\begin{remark*}
The same graphical construction that was given for the East process $\omega(t)$ applies to the process $\omega^\Front(t)$: this is clear for the East part of the generator $\cL^\East$; for the shift part $\cL^\Shift$,  simply apply a positive
shift $\vartheta^+$ when there is a ring at the origin and the
corresponding Bernoulli variable is one. If the Bernoulli variable is
zero, operate a negative shift $\vartheta^-$.
\end{remark*}

With this notation, the main result of Blondel~\cite{Blondel} can be summarized as follows.
\begin{theorem}[\cite{Blondel}]
\label{Blondel1} The front of the East process, $X(\omega(t))$, and the process as seen from the front, $\omega^\Front(t)$, satisfy the following:
  \begin{enumerate}[(i)]
  \item There exists a unique invariant measure
$\nu$ for the process   $\{\o^\Front(t)\}_{t\ge 0}$. Moreover, $\|\n-\pi\|_{(-\infty,-x]}$
decreases exponentially fast in $x>0$.
\item Let $q^*:=\nu(\o_{-1}=0)$ and let $v= q-pq^*$. Then $v>0$ and for any $\o\in \O_\Front$,
\[
\lim_{t\to \infty}\frac{X(\o(t))}{t}\
\stackrel{\bbP_\o}{\longrightarrow} \ v.
\]
  \end{enumerate}
\end{theorem}
Thus, if the East process has a front at time $t=0$ then it will have
a front at any later time. The latter progresses in time with an
asymptotically constant speed $v$.

\subsection{Local relaxation to equilibrium}
In this section we review the main technical results on the local
convergence to the stationary measure $\pi$ for the (infinite volume) East process. The key message here is
that \emph{each} vacancy in the starting configuration, in a time lag $t$,
induces the law $\pi$ in an interval in front of its position of length
proportional to $t$. That explains why the distance between the
invariant measure $\nu$ and $\pi$ deteriorates when we approach the
front from behind.
\begin{definition}
Given a configuration $\omega\in \Omega$ and an interval $I$ we say that $\o$
satisfies the \textbf{Strong Spacing Condition (SSC)} in $I$ if the largest sub-interval of
$I$ where $\omega$ is identically equal to one has length at most
$10\log |I|/(|\log p| \wedge 1)$.
Similarly, given $\d,\eps\in (0,1/4)$,  we will say that $\o$ satisfies the
\textbf{$(\d,\eps)$-Weak Spacing Condition (WSC)} in $I$ if the largest sub-interval
of $I$ where $\omega $ is identically equal to one has length at most
$\delta|I|^{\eps}$.
\end{definition}
For brevity, we will omit the $(\d,\eps)$ dependence in WSC case when these are made clear from the context.
\begin{proposition}
\label{prop:key1}
There exist universal positive constants $c^*,m$ independent
of $p$ such that the following holds. Let $\L=[1,2,\dots,\ell]$ and let $\o\in \O$ be such that $\o_0=0$.
Further let $\D(\o)$ be largest between the
  maximal spacing between two consecutive zeros of $\o$ in $\L$ and
  the distance of the last zero of $\o$ from the vertex $\ell$. Then
\[
\|\bbP^t_\o-\pi\|_{\L}\le \ell\left(c^*/q \right)^{\D(\o)} e^{-t
  \,(\gap(\cL)\wedge m)}.
\]
%\label{th:key2}
\end{proposition}
To prove this proposition, we need the following lemma.
\begin{lemma}
\label{infinitesupport} There exist universal positive constants $c^*,m$ independent
of $p$ such that the following holds. Fix $\o\in \O$ with $\o_0=0$,
let $\ell\in
\bbN$ and let $f:\O_{(-\infty,\ell\,]}\mapsto
\bbR$ with $\|f\|_\infty\le 1$. Let also $\pi_{\ell}(f)$ denote the new
function obtained by averaging $f$ w.r.t. the marginal of $\pi$ over the
spin at $x=\ell$. Then,
\begin{equation}
  \label{eq:fabio2}
|\E_{\omega}\left[f(\omega(t))-\pi_{\ell}(f)(\omega(t))\right]|\le
(c^*/q)^\ell e^{-t\, (\gap(\cL)\wedge m)}.
\end{equation}
\end{lemma}
\begin{remark*}
If we replaced the r.h.s.\ of \eqref{eq:fabio2} with $\left(2\sqrt{2}/(p\wedge q)\right)^\ell e^{-t\, \gap(\cL)}$, then the statement
would coincide with that in \cite{Blondel}*{Proposition
  4.3}. Notice that as $p \downarrow 0$, the term $c^*/q$ does not blow up---
  unlike $2\sqrt{2}/(p\wedge q)$---and as remarked below~\eqref{eq-AD-gap}, $\gap(\cL)$ stays bounded away from $0$. Hence, as $p\downarrow 0$, the time after which the r.h.s.\ in \eqref{eq:fabio2} becomes small is bounded from
above by $C_0\times \ell$ for some universal $C_0>0$ not depending on $p$. This fact will be crucially used in the proofs of some of the theorems to follow.
\end{remark*}
\begin{proof}[Proof of Lemma~\ref{infinitesupport}]
As mentioned in the remark using \cite{Blondel}*{Proposition
  4.3} it suffices to assume that
$p<1/3$. Fix $\o$ as in the lemma and let $\O^\o_{(-\infty,\ell\,]}$ be the set of all configurations
$\o'\in \O_{(-\infty,\ell\,]}$ which coincides with $\o$ on the half
line $(-\infty,0]$. The special configuration in
$\O^\o_{(-\infty,\ell\,]}$ which is identically equal to one in the
interval $[1,\ell]$ will be denoted by $\o^*$.
Observe that, using reversibility together with the fact that the
updates in $(-\infty,0]$ do not check the spins to the right of the
origin,
\begin{align}
\label{eq:fabio}
\sum_{\o'\in \O^\o_{(-\infty,\ell\,]}}\pi_{[1,\ell]}(\o')
\E_{\omega'}\left[f(\omega'(t)\right]
&=\E_{\omega}\left[\pi_{[1,\ell]}(f)(\omega(t)\right]\nonumber\\
\sum_{\o'\in \O^\o_{(-\infty,\ell\,]}}\pi_{[1,\ell]}(\o')
\E_{\omega'}\left[\pi_\ell(f)(\omega'(t)\right]
&=\E_{\omega}\left[\pi_{[1,\ell]}(f)(\omega(t)\right].
\end{align}
Using the
graphical construction as a grand coupling for the processes with
initial condition in $\O^\o_{(-\infty,\ell\,]}$, it is easy to verify
  that, at the hitting time $\tau_\ell$
of the set $\{\o'\in  \O_{(-\infty,\ell]}:\ \o'_\ell=0\}$ for the
process started from
$\o^*$, the processes starting from \emph{all} possible initial conditions in $\O^\o_{(-\infty,\ell\,]}$
have coupled.
Let $\o'\in\O^\o_{(-\infty,\ell\,]}$ be distributed according to $\pi_{[1,\ell]}.$
Then using the grand coupling,
\begin{align*}
  |\E_{\omega}\left[f(\omega(t))-\pi_{\ell}(f)(\omega(t))\right]| &=|\E_{\omega,\pi_{[1,\ell]}}\left[f(\omega(t))-f(\omega'(t))+\pi_{\ell}(f)(\omega'(t))-\pi_{\ell}(f)(\omega(t))\right]|\\
&\le 4 \sup_{\o'\in \O^\o_{(-\infty,\ell\,]}}\bbP(\exists \, x\in
  [1,\ell]:\ \o_x(t)\neq \o_x'(t))\\
&\le 4 \bbP_{\o^*}(\tau_\ell >t)\\
&\le 4 \bbP_{\o^*}(X(\o^*(t))<\ell).
\end{align*}
The first equality follows by adding and subtracting $\E_{\omega}\left[\pi_{[1,\ell]}(f)(\omega(t)\right]$ from the l.h.s.\ and then using  \eqref{eq:fabio}. The rest of the inequalities are immediate from the above discussion.
In order to bound the above probability, we observe that the front
$X(\o^*(t))$, initially at $x=0$, can be coupled to an asymmetric random
walk $\xi(t)$, with $q (\rm{resp.} p)$ as jump rate to the right(resp. left), in such a way that $X(\o^*(t))\ge \xi(t)$ for all $t\ge 0$. Since we have assumed that $p< 1/3,$
by standard hitting time estimates for biased random walk there exist universal constants $c,m$ such that, for $t\ge c
\ell$, the above probability is smaller than $e^{-mt}$.
\end{proof}
%\noindent
\begin{proof}[Proof of Proposition~\ref{prop:key1}]
Let $\o\in \O$ be such that $\o_0=0$.
Then
\begin{align*}
\max_{\substack{ f:\O_\L\mapsto \bbR \\
  \|f\|_\infty\le 1}}
  &|\bbE_\o\left[f(\omega(t))-\pi(f)\right]| \\
&\le  \max_{\substack{f:\O_\L\mapsto \bbR \\
  \|f\|_\infty\le
  1}}|\bbE_\o\left[f(\omega(t))-\pi_\ell(f)(\omega(t))\right]|+
\max_{\substack{f:\O_\L\mapsto \bbR \\ \|f\|_\infty\le 1}}|\bbE_\o\left[\pi_\ell(f)(\omega(t))-\pi(f)\right]|
  \\
&\le (c^*/q )^{\D(\o)} e^{-t\,
  (\gap(\cL)\wedge m)} +
\max_{\substack{f:\O_\L\mapsto \bbR \\\|f\|_\infty\le 1}}|\bbE_\o\left[\pi_\ell(f)(\omega(t))-\pi(f)\right]|,
\end{align*}
where we applied the above lemma to the shifted configuration in which
the origin coincides with the rightmost zero in $\L$ of $\o$. \\
We now
observe that the new function $\pi_\ell(f)$ depends only on the first
$\ell-1$ coordinates of $\o$ and that $\|\pi_\ell(f)\|_\infty\le
1$. Thus we can iterate the above bound $(\ell -1)$ times to get that
\begin{equation*}
\|\bbP_\o^t-\pi\|_\L\le 2\max_{\substack{f:\O_\L\mapsto \bbR \\
  \|f\|_\infty\le 1}}|\bbE_\o\left[f(\omega(t))-\pi(f)\right]| \le
\ell (c^*/q )^{\D(\o)} e^{-t\,(\gap(\cL)\wedge m)}.
\qedhere
\end{equation*}
  \end{proof}

  \begin{corollary}
\label{cor:spacing}
 Fix $\o\in \O^*\,$, $\ell \in \bbN$  and let
$I^\ell_\o=[X(\o),X(\o)+\ell-1]$. Then
\begin{align}
  \label{eq:1}
\sup_{\o\in \O^*}&\|\,\bbP^t_\o-\pi\|_{I^\ell_\o}\le (c^*/q )^{\ell} e^{-t\,
  (\gap(\cL)\wedge m)}.\\
\label{eq:2} \sup_{\o\in \O^*}&\bbP_\o\left(\o(t) \text{ does not satisfy SSC in $I^\ell_\o$}\right)
\le \ell (c^*/q )^{\ell} e^{-t\,
  (\gap(\cL)\wedge m)} +\ell^{-9}.\\
\label{eq:3}
\sup_{\o\in \O^*}&\bbP_\o\left(\o(t) \text{ does not satisfy WSC in
                   $I^\ell_\o$}\right)\le
(c^*/q )^{\ell} e^{-t\,
  (\gap(\cL)\wedge m)} + \ell p^{\d \ell^{\e/2}}.
\end{align}
\end{corollary}
  \begin{proof}
By construction, $\D_{I_\o^\ell}(\o)=\ell$ for any $\o\in \O^*$. Thus the first statement follows at once
from Proposition~\ref{prop:key1}. The other two statements  follow
from the
fact that
\[
\pi\left(\{\o: \text{$\o$ does not satisfy SSC
    in $[1,\dots,\ell]$\}}\right) \le \ell^{-9}
\]
and
\begin{equation*}
\pi\left(\{\o: \text{$\o$ does not satisfy the WSC
    in $[1,\dots,\ell]$\}}\right) \le \ell p^{\d \ell^{\e/2}}.
\qedhere
\end{equation*}
  \end{proof}
\subsection{Finite speed of information propagation}\label{sec:finite-speed}
As the East process is an interacting particle system whose rates are
bounded by one, it is well known that in this case information can
only travel through the system at finite speed. A quantitative statement of
the above general fact goes as follows.
\begin{lemma} \label{finitespeed} For $x<y\in \bbZ$ and $0\le s<t$, define
 the ``linking event''  $F(x,y;s,t)$ as the event that there exists  a ordered sequence
  $s\le t_{x}<t_{x+1}<\dots <t_y<t$ or $s\le t_{y}<t_{y-1}<\dots <t_x<t$ of rings of the Poisson clocks
  associated to the corresponding sites in $[x,y]\cap \bbZ$. Then there exists a constant $v_{\rm max}$ such that, for all $|y-x|\ge v_{\rm max}(t-s)$,
\[
\bbP(F(x,y;s,t))\le e^{-|x-y|}.
\]
\end{lemma}
\begin{proof}The probability of $F(x,y;s,t)$ is equal to the
  probability that a Poisson process of intensity $1$ has at least
  $|x-y|$ instances within time $t-s$.
\end{proof}
\begin{remark}
\label{rem:speed}
An important consequence of the above lemma is the following
fact. Let $0<s<t$ and let $\cF_s$ be the $\s$-algebra
generated by all the rings of the Poisson clocks and all the coin
tosses up to time $s$ in the graphical construction of the East
process. Fix $x<y<z$ and let $A,B$ be two events depending on
$\{\o_a\}_{a\le x}$ and $\{\o_a\}_{a\ge z}$ respectively. Then
\begin{gather*}
\bbP_\o\left(\{\o(t)\in A\cap B\}\cap F(y,z;s,t)^c \mid
  \cF_s\right)\\
=\bbP_\o\left(\{\o(t)\in A\}\mid \cF_s\right)
\bbP_\o\left(\{\o(t)\in B\}\cap F(y,z;s,t)^c\mid \cF_s\right).
 \end{gather*}
This is because: (i) on the event $F(y,z;s,t)^c$ the occurrence of the
event $B$ does not depend anymore on the Poisson rings and coin tosses to the
left of $y$; (ii) the occurrence of the event $A$ depends only on
the Poisson rings and coin tosses to the left of $x$ because of the
oriented character of the East process.
\end{remark}
The finite speed of information propagation, together with the results of~\cite{AD02}, implies the following rough bound on the position of the front
$X(\o(t))$ for the East process started from $\o\in \O^*$ (also see, e.g., \cite{Blondel}*{Lemma~3.2}).
\begin{lemma}
\label{linearspeed}
There exists constants $v_{\rm min}>0$ and $\gamma>0$ such that
\[
\sup_{\o\in \O^*}\bbP_\o\left(X(\o(t))\in[X(\o)+ v_{\rm min}t,X(\o)+v_{\rm max}t]\right) \ge
1-e^{-\gamma t}.
\]
\end{lemma}
\begin{remark}\label{speedbnd} When $p\downarrow 0$ one can obtain the above statement with $v_{\rm min}\to 1$ and $\gamma$ uniformly bounded away from $0$ by using our Proposition~\ref{prop:key1} instead of \cite{Blondel}*{Proposition~4.3} in the proof of \cite{Blondel}*{Lemma~3.2}.
\end{remark}
The second consequence of the finite speed of information propagation
is a kind of mixing result behind the front $X(\o(t))$ for the process started from
$\o\in \O^*$.  We first need few additional notation.
\begin{definition}
\label{shift}For any $a\in \bbZ$, we define the
\emph{shifted} configuration $\vartheta_a \o$ by
\[
\vartheta_a \o_x= \o_{x+a},\ \forall x\in \bbZ.
\]
\end{definition}
\begin{proposition}
\label{prop:decor}
Let $\L\subset (-\infty, -\ell]\cap \bbZ$ and let
$B\subset \{0,1\}^{\L}$.
Assume $\ell \ge 2v_{\rm max}(t-s)$. Then for any $\o\in \O^*$ and any
$a\in \bbZ$ the following holds:
\begin{gather*}
\Big|\,\bbE_\o\left[{\mathds 1}_{\{\vartheta_{X(\o(s))} [\o(t)]\in B\}}{\mathds
    1}_{\{X(\o(t))=a\}}\mid \cF_s\right]-\bbE_\o\left[{\mathds 1}_{\{\vartheta_{X(\o(s))} [\o(t)]\in B\}}\mid \cF_s\right]\bbE_\o\left[{\mathds 1}_{\{X(\o(t))=a\}}\mid \cF_s\right]\,\Big|\\= O(e^{-\ell}).
\end{gather*}
\end{proposition}
To see what the proposition roughly tells we first assume that the front at time $s$ is at $0$. Then the above result says
that at a later time $t$  any event supported on $(-\infty, -\ell]$ is almost independent of the location of the front.
\begin{proof}
Recall the definition of the event $F(x,y;s,t)$ from Lemma~\ref{finitespeed} and let
\begin{align*}
B_1 &:=F\left(X(\o(s))-\ell,X(\o(s))-\ell/2-1;s,t\right)\\
B_2&:= F\left(X(\o(s))-\ell/2,X(\o(s));s,t\right).
\end{align*}
We now write
\begin{align*}
{\mathds 1}_{\{\vartheta_{X(\o(s))} [\o(t)]_\L\in B\}}{\mathds
    1}_{\{X(\o(t))=a\}}
&={\mathds 1}_{\{\vartheta_{X(\o(s))} [\o(t)]_\L\in B\}}{\mathds
    1}_{\{X(\o(t))=a\}}{\mathds 1}_{\{B_1^c\}}{\mathds 1}_{\{B_2^c\}} \\
&+{\mathds 1}_{\{\vartheta_{X(\o(s))} [\o(t)]_\L\in B\}}{\mathds
    1}_{\{X(\o(t))=a\}}\left[1-{\mathds 1}_{\{B_1^c\}}{\mathds 1}_{\{B_2^c\}}\right].
\end{align*}
We first note that given $\cF_s$ for any $a< X(\o(s))-\ell/2$,
\begin{align*} {\mathds
    1}_{\{X(\o(t))=a\}}{\mathds
    1}_{\{B_2^c\}}=0,
\end{align*}
and hence
\begin{align*} \bbE_\o\left[{\mathds
    1}_{\{X(\o(t))=a\}}{\mathds
    1}_{\{B_2^c\}}\mid \cF_s\right]=0.
\end{align*}
Thus, we may assume that $a\ge X(\o(s))-\ell/2$.
Now
\begin{align*}
\bbE_\o&\left[{\mathds 1}_{\{\vartheta_{X(\o(s))} [\o(t)]_\L\in B\}}{\mathds
    1}_{\{X(\o(t))=a\}}{\mathds 1}_{\{B_1^c\}}{\mathds
    1}_{\{B_2^c\}}\mid \cF_s\right]\\
&=\bbE_\o\left[{\mathds 1}_{\{\vartheta_{X(\o(s))} [\o(t)]_\L\in B\}}{\mathds
    1}_{\{B_1^c\}}\mid \cF_s\right]
\bbE_\o\left[{\mathds
    1}_{\{X(\o(t))=a\}}{\mathds
    1}_{\{B_2^c\}}\mid \cF_s\right]
 \end{align*}
because under the assumption that $a\ge X(\o(s))-\ell/2$,  the two events are functions of an independent set of
variables in the graphical construction (cf.\ Remark~\ref{rem:speed}).
By Lemma~\ref{finitespeed} we know that $\bbP(B^c_i\mid \cF_s)\le
e^{-\ell},\ i=1,2$ and the proof is complete.
\end{proof}

\section{The law behind the front of the East process}\label{sec:front}

Our main result in this section is a quantitative estimate on the rate of
convergence as $t\to \infty$ of the law $\mu_\o^t$ of the process seen from  the front
to its invariant measure $\nu$.
Consider the process $\{\o^\Front(t)\}_{t\ge 0}$ seen from the front (recalling \S\ref{sec:proc-front}) and let $\mu_\o^t$ be its law at time $t$ when the starting configuration
is $\o$.

\begin{theorem}
\label{coupling}
For any $p\in (0,1)$ there exist $\a\in (0,1)$ and $v^*>0$ such that
\[
\sup_{\o\in \O_\Front}\|\,\mu_\o^t -\nu\|_{[-v^*t,\,0]}=O(e^{-t^\a}).
\]
Moreover, $\a$ and $v^*$ can be chosen uniformly as $p\to 0$.
\end{theorem}

  A corollary of this result --- which will be key in the proof of Theorem~\ref{th:main1} --- is to show that, for
any $\o\in \O_\Front$, the increments in the position of the front
(the variables $\xi_n$ below) behave asymptotically as a
stationary sequence of weakly dependent random variables with
exponential moments.

Fix $\D\footnote{In the sequel we will always use the letter $\D$ to denote a time lag. Its value will depend on the context and will be specified in advance. }>0$ and let $t_n=n\D$ for $ n\in \bbN$. Define
\[
\xi_n:= X(\o(t_n))-X(\o(t_{n-1})),
\]
so that
\begin{equation}
  \label{eq:19}
X(\o(t))=\sum_{n=1}^{N_t} \xi_n + \left[X(\o(t))-X(\o(t_N))\right], \quad N=
\lfloor t/\D\rfloor.
\end{equation}
%In what follows we will denote by $\bbE_\nu[\cdot], \var_\nu(\cdot)$ the
%expectation and variance respectively w.r.t. the stationary process
%started from $\O_\Front$ with initial law $\nu$.
Recall also that
$\a,v^*$ are the constants appearing in Theorem~\ref{coupling}.
\begin{corollary}
\label{cor:wf}\ Let $f:\bbR\mapsto [0,\infty)$ be such that $e^{-|x|}f^2(x)\in L^1(\bbR)$. Then
  \begin{equation}
    \label{eq:20}
C_f\equiv \sup_{\o\in
  \O_\Front}\bbE_\o\left[f(\xi_1)^2\right]<\infty.
  \end{equation}
Moreover, there exists a constant $\g>0$ such that
\begin{align}
\label{eq:20tris}
\sup_{\o\in \O_\Front}|\bbE_\o\left[f(\xi_n)\right]- \bbE_\nu\left[f(\xi_1)\right]|=
O(e^{-\g n^\a}) \quad \forall n\ge 1,
 \end{align}
\begin{gather}
\label{eq:13}
\sup_{\o\in \O_\Front}|{\cov}_\o\left(\xi_j,\xi_n \right)-{\cov}_\nu\left(\xi_1,\xi_{n-j}\right)|= O(e^{- \g j^\a})\wedge
O(e^{-\g(n-j)^\a})\quad \forall j<n
\end{gather}
and
\begin{gather}
  \label{eq:13fabio}
  \sup_{\o\in \O_\Front}|{\cov}_\o\left(f(\xi_j),f(\xi_n)
  \right)|=O(e^{-\gamma (n-j)^\a}), \quad \ \forall j\le n,
\end{gather}
where the constants in the r.h.s.\ of \eqref{eq:20tris} and
\eqref{eq:13fabio} depend on $f$ only through the constant $C_f$.
Finally, for any $k,n \in \bbN$ such that $v^*k>n v_{\rm max}$ and for any bounded $F:\bbR^n\mapsto \bbR$ ,
  \begin{equation}
    \label{eq:21}
\sup_{\o\in \O_\Front}\Big|
\bbE_\o\left[F\Bigl(\xi_{k},\xi_{k+1},\dots,\xi_{k+n-1}\Bigr)\right] -
\bbE_\nu\left[F\Bigl(\xi_{1},\xi_{2},\dots,\xi_{n}\Bigr)\right]\Big|=
O\left(e^{-\g t_k^\a}\right).
  \end{equation}
%\label{lem:wd}
\end{corollary}

To prove Theorem~\ref{coupling} we will require a technical result, Theorem~\ref{th:key3} below, which can informally be summarized as follows:
\begin{itemize}
  \item Starting from $\o\in \O^*$, at any fixed large time $t$, with high probability the configuration satisfies
WSC apart from an interval behind the front $X(\o(t))$ of length proportional to
$t^\eps$ .
\item If the above property is true at time $t$, then at a later
time $t'=t+ \text{const}\times t^\eps$ the law of the process will be very
close to $\pi$ apart from a small interval behind the front
where the strong spacing property will occur with high probability.
\end{itemize}
Formally, fix a constant $\kappa$ to be chosen later on and $t>0$. Let
$\ell\equiv t^{\eps}$, where $\eps$ appears in the WSC and let $t_\ell=t-\kappa\ell/v_{\rm
  min}$. Let $\cS_\ell$ denotes the set of
configurations  which fail to satisfy SSC
in the interval $[-3(v_{\rm max}/v_{\rm min})\kappa\,\ell,- \kappa \log \ell)\cap \bbZ$ and let
$\cW_{\ell,t}$ be those
configurations which fail to satisfy WSC
in the interval $[-v_{\rm min}t,- \ell)\cap \bbZ$.
\begin{theorem}
\label{th:key3}
It is possible to choose $\d$ small enough and $\kappa$ large enough depending only on $p$ in such a way
that for all $t$ large enough the following
holds:
\begin{align}
\label{1} &\sup_{\o\in \O_\Front}\m^t_\o\left(\cW_{\ell,t}\right) = O(e^{-t^{\eps/2}}),\\
\label{2}
&\sup_{\o\in \O_\Front}\m^t_\o\left(\cS_\ell \mid \cF_{t_\ell}\right)= O(t^{-7\eps}) +
{\mathds 1}_{\cW_{\ell,t}}(\o(t_\ell)),\\
\label{3} &\sup_{\o\in \O_\Front}\|\mu_\o^t(\cdot \mid
\cF_{t_\ell})-\pi\|_{[-v_{\rm min}t,-3(v_{\rm max}/v_{\rm min})\kappa\ell]}= O(e^{-t^{\eps/2}}) +
{\mathds 1}_{\cW_{\ell,t}}(\o(t_\ell)).
  \end{align}
  Moreover, $\kappa$ stays bounded as $p\downarrow 0.$
\end{theorem}

\subsection{Non-equilibrium properties of the law behind the front: Proof of Theorem~\ref{th:key3}}
We begin by proving \eqref{1}. Bounding $\sup_{\o\in
  \O_\Front}\m^t_\o\left(\cW_{\ell,t}\right)$ from above is equivalent to bounding $\sup_{\o\in
  \O_\Front}\bbP_\o(\o(t)\in \cW^*_{\ell,t})$ from above, where
$\cW^*_{\ell,t}$ denotes the set of configurations $\o\in \O^*$ which do not
satisfy the
    spacing condition in $[X(\o)-v_{\rm min}t,X(\o)-\ell]$.

Using Lemma~\ref{linearspeed}, with probability greater than $1- e^{-\g
  t}$ we can assume that $X(\o(t))\in [v_{\rm min}t, v_{\rm
  max}t]$.
Next we observe that, for any $a\in [v_{\rm min}t, v_{\rm
  max}t]$, the events  $\{X(\o(t))=a\}$ and $\{\o(t)\in
\cW^*_{\ell,t}\}$ imply that there exists $x\in \bbZ$ with the following properties:
\begin{itemize}
\item \hskip 1cm $0\le x \le  a-\ell$;
\item \hskip 1cm The hitting time $\tau_x:=\inf\{s>0: X(\o(s))=x\}$ is smaller
than $t$;
\item \hskip 1cm $\o(t)$ is identically equal to one in the interval $I_x:=[x,x+\d (v_{\rm min}\,t)^{\eps}/2]$;
\item \hskip 1cm The linking event $F(x,a;\t_x,t)$ defined in Lemma~\ref{finitespeed}
occurred.
\end{itemize}
In conclusion, using twice a union bound (once for the
choice of $a\in [v_{\rm min}t, v_{\rm
  max}t]$ and once for the choice of $x\in [0,a-\ell]$) together with the strong
Markov property at time $\t_x$, we get
\begin{align*}
    & \bbP_\o(\o(t)\in \cW^*_{\ell,t}) \\
\quad & \le e^{-\g t} + \sum_{a=v_{\rm min}t}^{v_{\rm max}t}\
  \sum_{x=0}^{a-\ell } \    \bbP\left(F(x,a;\t_x,t)\right){\mathds
    1}_{\{|x-a|\ge v_{\max}(t-\t_x)\}} \\
& \quad \qquad \, + \sum_{a=v_{\rm min}t}^{v_{\rm max}t}\  \sum_{x=0}^{a-\ell } \left[
 \|\bbP^t_{\o(\t_x)}-\pi\|_{I_x} + p^{\frac{\d}{2}( v_{\rm min}t)^{\eps}}\right]{\mathds 1}_{\{|x-a|\le v_{\rm  max}(t-\t_x)\}} \\
\quad &\le (v_{\rm \max}t)^2 \bigg[
 e^{-\g t}+ 2e^{-\ell}+  \sqrt{2}{{\d}( v_{\rm min}t)^{\eps}}\,\left(\frac{c^*}{ q
    }\right)
^{\frac{\d}{2}( v_{\rm min}t)^{\eps}} e^{-\frac{\ell}{v_{\rm max}}
  \,(\gap(\cL)\wedge m)}+ p^{\frac{\d}{2}( v_{\rm min}t)^{\eps}}\bigg].
\end{align*}
Above we used Lemma~\ref{finitespeed} in the case $
|x-a|\ge v_{\rm max}(t-\t_x)$ and \eqref{eq:1} of Corollary~\ref{cor:spacing}
otherwise. The statement \eqref{1} now follows by taking $\d$ small enough.

\medskip

We now prove \eqref{2}. As before we give the result in the East
process setting (\ie for the law $\bbP_\o^t(\cdot\mid \cF_s)$ and
$\cS_{\ell}$ replaced by its random shifted version
$\cS^*_{\ell}$). We decompose the interval $[X(\o(t))-3(v_{\rm max}/v_{\rm
  min})\kappa\,\ell, X(\o(t)) -\kappa \log \ell)]\cap \bbZ$ where we want SSC to
hold into $[X(\o(t_\ell)),X(\o(t))-\kappa \log \ell]$ and $[X(\o(t))-3(v_{\rm
  max}/v_{\rm min})\kappa\,\ell,X(\o(t_\ell))].$\\
  Note that by Lemma~\ref{linearspeed} we can ignore the events $\{X(\o(t_\ell))>X(\o(t))-\kappa \log \ell\}$ and  $\{X(\o(t))-3(v_{\rm
  max}/v_{\rm min})\kappa\,\ell > X(\o(t_\ell))\}.$ \\

  \noindent
We now proceed in two steps: \begin{inparaenum}[(1)]
    \item we show that
SSC occurs with high probability in the first interval. Here we do not use the condition that
$\o(t_\ell)\notin \cW^*_{\ell,t}$.
\item we prove the same
statement for the second
interval. Here instead the fact that $\o(t_\ell)\notin
\cW^*_{\ell,t}$ will be crucial.
  \end{inparaenum}
\\

- {\it Step (1).} Let $\D\equiv 5\log \ell/(|\log p|\wedge 1)$. For any
intermediate time $s\in [t_\ell, t -(\kappa/v_{\rm
  max})\log \ell]$,
Corollary~\ref{cor:spacing} together with the Markov property at time
$s$ show that
\begin{align}
\bbP_\o&\left(\o(t)_x=1\ \forall x\in [X(\o(s)),X(\o(s))+\D]\mid
  \cF_s\right)\nonumber\\
&\le \left\|\bbP_\o(\cdot\mid \cF_s)-\pi\right\|_{[X(\o(s)),X(\o(s))+\D]}+
\pi\left(\o_x=1\ \forall x\in [X(\o(s)),X(\o(s))+\D]\right)
\nonumber\\
\label{eq:4}
&\le \D\left(\frac{c^*}{q}\right)^\D e^{-(t-s)(\gap(\cL)\wedge m)} +p^{|\D|}=O(t^{-10\eps}).
\end{align}
Above we used the fact that $t-s\ge \kappa/v_{\rm max}\log
\ell$. Hence,  $\kappa$ can be chosen depending only on
$p$ such that  \eqref{eq:4} holds and $\kappa$ stays bounded as $p\downarrow 0.$\\

\noindent
We now take the union of the random intervals $[X(\o(s)),X(\o(s))+\D]$
over discrete times $s$ of the form $s_j= t_\ell+ j/\ell^2$,
$j=0,1,\dots, n$ and $n$ such that $s_n=t -(\kappa/v_{\rm max})\log
\ell$. Thus $n=O(\ell^3)$=$O(t^{3\eps})$. The aim here is to show
that, with high probability,
the above union is actually an interval containing the target one $[X(\o(t_\ell)),
X(\o(t))-\kappa \log \ell]$, with the additional property that it does not contain
a sub-interval of length $\D$ where $\o(t)$ is constantly equal to one (which will then imply~\eqref{2}, with room to spare).\\

\noindent
We now upper bound the probability that the set
$\cup_{j=0}^n[X(\o(s_j)),X(\o(s_j))+\D]$ is not an interval. It is an easy observation that if $X(\o(s_n))>X(\o(s_0))$ then the aforementioned event occurs if $X(\o(s_{j+1}))-X(\o(s_j))\le \D$ for $j=0,1,\ldots n.$
Now by the lower bound in Lemma~\ref{linearspeed}
$$\bbP(X(\o(s_n))>X(\o(s_0)))\ge 1-e^{-ct^{\eps}}$$ for some constant $c.$
Also
\begin{align*}
\sum_j \bbP_\o&(X(\o(s_{j+1}))-X(\o(s_j))\ge \D)\\
&\le \sum_j\E\left[
\bbP_\o\left(F(X(\o(s_j)),X(\o(s_j))+\D; s_j,s_{j+1})\mid
  \cF_{s_j}\right)\right]\\
&\le n e^{-\D} = O(t^{-8\eps}).
\end{align*}
Above $F(X(\o(s_j)),X(\o(s_j))+\D; s_j,s_{j+1})$ is the linking event
and we used Lemma~\ref{finitespeed} because $\D\gg (s_{j+1}-s_j)$.

Moreover, Lemma~\ref{linearspeed} implies that $\kappa$ can be chosen (bounded as $p\downarrow 0$), such that with probability greater than
\[
1-e^{-\g (t-s_0)}-e^{-\g (t-s_n)}=1-O(t^{-10 \eps}),
\]
the front $X(\o(t))$ satisfies
\begin{align*}
X(\o(t))\le X(\o(s_n))+v_{\rm max}(t-s_n)\le X(\o(s_n))+\kappa\log
\ell.
\end{align*}
Thus
\[
[X(\o(t_\ell)),X(\o(t))-\kappa \log \ell] \subset \cup_{j=0}^n[X(\o(s_j)),X(\o(s_j))+\D].
\]
with probability $1-O(t^{-10 \eps})$.

Finally, using \eqref{eq:4} and union bound, the probability that  there exists $j\le n$ such that
$\o(t)$ is identically equal to one in $[X(\o(s_j)),X(\o(s_j))+\D]$ is
$O(t^{-7\eps})$ uniformly in the configuration at time $t_\ell$.

In conclusion we proved that SSC holds with probability $1-
O(t^{-8\eps})$ in an interval containing $[X(\o(t_\ell)),
X(\o(t)-\kappa \log \ell]$. The first step is complete.\\

- {\it Step (2).} Let $x^*=\max\{x\le  X(\o(t_\ell))- 3\kappa (v_{\rm
  max}/v_{\rm min})\ell: \ \o(t_\ell)_x=0\}$. Since $\o(t_\ell)\notin
\cW^*_{\ell,t}$ such a zero exists.
Moreover, $\o(t_\ell)\notin \cW^*_{\ell,t}$ implies that
$\o(t_\ell)$ has a zero in every sub-interval of $[x^*, X(\o(t_\ell))-
\ell]$ of length $\d t^\eps=\d \ell$. Hence we can apply Proposition
\ref{prop:key1} to the interval $[x^*, X(\o(t_\ell))]$ to get that
\[
\|\,\bbP^t_\o(\cdot \mid \cF_{t_\ell})-\pi\|_{[x^*, X(\o(t_\ell))]}=
O(e^{- t^{\eps/2}}),
\]
by choosing $\kappa$ large enough. Since by Remark~\ref{speedbnd} $v_{\rm min}\rightarrow 1$ as $p \downarrow 0$, we can choose $\kappa$ to be bounded as $p \downarrow 0$.
Also
\[
\pi\left(\o:\ \o \text{
  violates SSC in }[x^*, X(\o(t_\ell))]\right)= O(t^{-7\eps}).
\]
Thus we have proved that SSC holds in $[x^*, X(\o(t_\ell))]$ with probability $1-O(t^{-7\eps})$.

Finite speed of propagation in the
form of Lemma~\ref{linearspeed} guarantees that, with probability $1- O(e^{-\g
  (t-t_\ell)})$, $x^* < X(\o(t))-2\kappa (v_{\rm max}/v_{\rm
  min})\ell$.
The proof of \eqref{2} is complete.

\medskip
It remains to prove \eqref{3}. Let $\L:= [-v_{\rm
  min}t,-3(v_{\rm max}/v_{\rm min})\kappa\ell]\cap \bbZ$ and let $A\subset \{0,1\}^{\L}$. Recall
Definition~\ref{shift} of the shifted configuration $\vartheta_a \o$ and that $t_{\ell}=t-\kappa \ell/{v_{\rm min}}$.
Then \eqref{3} follows once we show that
\[
|\bbP_\o(\vartheta_{X(\o(t)} \o(t)_\L\in A\mid \cF_{t_\ell})-\pi(A)|
\le e^{-t^{\eps/2}}
\]
whenever $\o(t_\ell)$ satisfies WSC in the interval
$I=[X(\o(t_\ell))-v_{\rm min}t, X(\o(t_\ell))-\ell]$. This property is
assumed henceforth.
Let us decompose $\bbP_\o(\vartheta_{X(\o(t)} \o(t)_\L\in A\mid
\cF_{t_\ell})$ according to the value of the front:
\begin{gather*}
  \bbP_\o\left(\vartheta_{X(\o(t))} \o(t)\in A\mid
\cF_{t_\ell}\right)=
\sum_{a\in \bbZ}\bbE_\o\left[{\mathds 1}_{\{\vartheta_{a}
\o(t)_\L\in A\}}\,{\mathds 1}_{\{X(\o(t))=a\}}\mid \cF_{t_\ell}\right].
\end{gather*}
Using Lemma~\ref{linearspeed},  $0<X(\o(t))-X(\o(t_\ell))\le
v_{\rm max}(t-t_\ell)$ occurs with probability
greater than $1-e^{-\g(t-t_\ell)}$. Thus
\begin{gather*}
\sum_{a\in \bbZ}\bbE_\o\left[{\mathds 1}_{\{\vartheta_{a}
\o(t)_\L\in A\}}\,{\mathds 1}_{\{X(\o(t))=a\}}\mid
\cF_{t_\ell}\right]\\
=
\sum_{\substack{ a\in \bbZ \\ 0<a-X(\o(t_\ell))\le v_{\rm max}(t-t_\ell)}}\bbE_\o\left[{\mathds 1}_{\{\vartheta_{a}
\o(t)_\L\in A\}}\,{\mathds 1}_{\{X(\o(t))=a\}}\mid
\cF_{t_\ell}\right] + e^{-\g(t-t_\ell)}.
\end{gather*}
By definition, the event $\{\vartheta_{a}
\o(t)_\L\in A\}$  is the same as the event $\{
\o(t)_{\L+a}\in A\}$.
Using the restriction  that $|a-X(\o(t_\ell))|\le v_{\rm max}(t-t_\ell)$, the choice of $\L$ and the fact that $(v_{\rm max}/v_{\rm min})\kappa\ell \ge  v_{\rm max}(t-t_\ell)$, we get $\L+a \subset (-\infty,X(\o(t_\ell))-2( v_{\rm max}/v_{\rm min})\kappa\ell]$. Thus, the event $\{
\o(t)_{\L+a}\in A\}$ satisfies the hypothesis of  Proposition~\ref{prop:decor}, which can then be applied to
each term in the above sum to
get
\begin{align*}
&\sum_{\substack{a\in \bbZ \\ 0<a-X(\o(t_\ell))\le v_{\rm max}(t-t_\ell)}}\!\!\!\bbE_\o\left[{\mathds 1}_{\{\vartheta_{a}
\o(t)_\L\in A\}}\,{\mathds 1}_{\{X(\o(t))=a\}}\mid \cF_{t_\ell}\right]
\\
=&\sum_{\substack{a\in \bbZ \\ 0<a-X(\o(t_\ell))\le v_{\rm max}(t-t_\ell)}}
\!\!\!\bbE_\o\left[{\mathds 1}_{\{\vartheta_{a}
\o(t)_\L\in A\}}\mid \cF_{t_\ell}\right]\,\bbE_\o\left[{\mathds
1}_{\{X(\o(t))=a\}}\mid \cF_{t_\ell}\right] +O(\ell \,e^{-\ell}).
   \end{align*}
Finally we claim that, for any $a$ such that $0<a-X(\o(t_\ell))\le v_{\rm
  max}(t-t_\ell)$, if $\d$ is chosen small enough and $\kappa$ large
enough depending on $p$ (bounded as $p \downarrow 0$),
\begin{equation}
  \label{eq:any a}
\bbE_\o\left[{\mathds 1}_{\{\vartheta_{a}
\o(t)_\L\in A\}}\mid \cF_{t_\ell}\right] = \pi(A) + O(e^{-t^{\eps/2}}).
 \end{equation}
To prove it we apply Proposition~\ref{prop:key1} to the interval $I=[X(\o(t_\ell))-v_{\rm min}t, X(\o(t_\ell))-\ell]$) to get that
\begin{equation}
  \label{eq:5}
\|\bbP^t_\o(\cdot\mid \cF_{t_\ell})-\pi\|_I \le
|I|\left(\frac{c^*}{ q}\right)^{\d |I|^\eps}
e^{-(t-t_\ell)(\gap(\cL)\wedge m)},
\end{equation}
where $|I|=O(t)$ is the length of $I$, since by assumption
$\o(t_\ell)$ satisfies WSC in $I$. Because of our choice of the
parameters $(\ell, t_\ell)$ the r.h.s.\ of \eqref{eq:5} is
$O(e^{-t^{\eps/2}})$ if $\d,\kappa$ are chosen small enough and large enough
respectively depending on $p$. Since by Remark~\ref{speedbnd} $v_{\rm min}\rightarrow 1$ as $p \downarrow 0,$ $\kappa$ can be chosen to be bounded as $p \downarrow 0.$\\

\noindent
The claim now follows because $\{\o:
\vartheta_a\o \in A\}\subset \{0,1\}^{\L+a}$, with
\begin{align*}
\L+a &=[-v_{\rm min}t +a, -3(v_{\rm max}/v_{\rm min})\kappa\ell +a]\\
&\subset [X(\o(t_\ell))-v_{\rm min}t_\ell,\, X(\o(t_\ell))- (v_{\rm max}/v_{\rm
  min})\kappa\ell\,] \subset I,
\end{align*}
together with the translation invariance of $\pi$ expressed by $\pi\left(\{\o: \vartheta_a\o \in A\}\right)=\pi(A)$.
This establishes \eqref{3} and concludes the proof of Theorem~\ref{th:key3}.
Notice that at all points in the proof, $\kappa$ was chosen to be bounded as $p \downarrow 0.$
\qed

\subsection{On the rate of convergence to
  the invariant
  measure $\nu$: Proof of Theorem~\ref{coupling}}
The proof is based on a coupling argument. There exists $v^*>0$ such
that, for any $t$ large enough
 and for any pair of starting
configurations $(\o,\o')\in \O_\Front\times \O_\Front$,
\begin{equation}
  \label{eq:6}
  \|\mu_\o^t-\mu_{\o'}^t\|_{[-v^*t,\,0]} \le c' e^{-t^\a},
\end{equation}
with $(c',\a)$ independent of $(\o,\o')$. Also $v^*, \alpha $ can be chosen uniformly as $p\downarrow 0.$ Once this step is
established and using the invariance of the measure $\nu$ under the
action of the semigroup $e^{t\cL^\Front}$,
\begin{align*}
  \|\,\mu_\o^t -\nu\|_{[-v^*t,\,0]}&=
\|\,\mu_\o^t -\int d\nu(\o')\mu_{\o'}^t\,\|_{[-v^*t,\,0]}\\
&\le \int d\nu(\o')\|\,\mu_\o^t -\mu_{\o'}^t\,\|_{[-v^*t,\,
  0]}\le c' e^{-t^\a}.
\end{align*}
We now prove \eqref{eq:6}. We first fix a bit of notation.

Given $\eps\in (0,1)$ and a large $t>0$, let $\D_1=(\kappa/v_{\rm
  min}) t^\eps$ where $\kappa$ is the constant appearing in Theorem
\ref{th:key3}, let
$\D_2=\kappa\eps\log t$ and define $\D=\D_1+\D_2$. We then set
\[
t_0=(1-\eps)t,\quad t_n=t_{n-1}+\D, \quad n=1,\dots N,\quad  N=
\lfloor\eps\, t/\D\rfloor.
\]
It will be
convenient to refer to the time lag $[t_{n-1},t_n)$ as the
$n^{\rm th}$-round. In turn we split each round into two parts: from
$t_{n-1}$ to $s_n:= t_{n-1}+\D_1$ and from $s_n$ to $t_n$. We will
refer to the first part of the round as the \emph{burn-in part} and to
the second part as the \emph{mixing part}. We also set
$I_n=[-v_{\rm min}t_n+ 2v_{\rm max} \D n, 0]$. Observe that $I_n\neq
\emptyset$ for any $n\le N+1$ if $\eps$ is chosen smaller
than $v_{\rm min}/v_{\rm max}$ and $t$ is large enough depending on $\eps$.

Next, for any pair $(\mu,\mu')$ of probability
measures on a finite set, we denote by ${\rm MC}(\mu,\mu')$ their
\emph{maximal coupling}, namely the one that achieves the variation
distance between $\mu,\mu'$ in the variational formula (see, e.g.,~\cite{LPW})
\[
\|\mu-\mu'\|=\inf\{M(\o\neq \o'):\ M \text{ a coupling of } \mu,\mu'\}.
\]
If $(\mu,\mu')$ are probability
measures on $\O$ and $\L$ is a finite subset of
$\bbZ$, we define the \emph{
  $\L$-maximal coupling} ${\rm MC}_\L(\mu,\mu')$  as follows:
\begin{enumerate}[a)]
\item first sample $(\o_\L,\o_\L')$ according to
the maximal coupling of the marginals of $\mu,\mu'$ on $\O_\L$;
\item then sample \emph{independently}
$(\o_{\bbZ\setminus \L}, \o'_{\bbZ\setminus \L})$ according to their
respective conditional distribution $\mu(\cdot\mid \o_\L), \mu'(\cdot\mid \o'_\L)$.
\end{enumerate}
Finally the \emph{basic coupling} for the
East process will be the one in which two configurations evolve
according to the graphical construction using the same Poisson clocks
and the same coin tosses.

We are now ready to recursively construct the coupling $M_{\o,\o'}^t$ of
$\mu_\o^t,\mu_{\o'}^t$ satisfying \eqref{eq:6}. For lightness of
notation, in the sequel the starting configurations $(\o,\o')$ will be
sometimes omitted.  \\
\begin{definition}[The coupling $M_{\o,\o'}^t$] We first define a
  family $\{M^{(n)}\}$ of couplings for
  $\{\left(\mu_\o^{t_n},\mu^{t_n}_{\o'}\right)\}_{n=0}^N$ as follows.
$M^{(0)}$ is the trivial product coupling. Given $M^{(n)}$, the coupling $M^{(n+1)}$ at time $t_{n+1}$ is constructed
according to the following algorithm:
  \begin{enumerate}[(a)]
  \item Sample $(\o(t_n),\o'(t_n))$ from $M^{(n)}$. If they
    coincide in the interval $I_n$ then let them evolve according
    to the basic coupling for a time lag $\D$;\\
\item otherwise, sample $(\o(s_n),\o'(s_n))$ at the
  end of the burn-in part of round$(n+1)$ via the
  $\L_n$-maximal coupling $MC_{\L_n}$ for the laws
  $\mu_\o^{s_n}(\cdot\mid \cF_{t_n})$ and $\mu_{\o'}^{s_n}(\cdot\mid\cF_{t_n})
  $ at the configurations $(\o(t_n),\o'(t_n))$ from step (a). Here $\L_n=[-v_{\rm min}s_n, -3\left(v_{\rm max}/v_{\rm
      min}\right)\kappa t^\eps]$.\\
  \begin{enumerate}[(i)]
  \item If $(\o(s_n),\o'(s_n))$ are not equal in the interval $\L_n$,
    then let them evolve for the mixing part of the round (i.e., from
    time $s_n$ to
    time $t_{n+1}$) via
    the basic coupling.\\
\item  If instead they agree on $\L_n$, then search for the rightmost common zero of
  $(\o(s_n),\o'(s_n))$ in
  $\L_n$ and call $x_*$ its position. If there is no such a zero,
  define $x_*$ to be the right boundary of $\L_n$. Next sample a
  Bernoulli random variable $\xi$ with
  $\text{\rm Prob}(\xi=1)=e^{-2\D_2}$. The value $\xi=1$ has to be
  interpreted as corresponding to the event that the two Poisson clocks
  associated to $x_*$ and to the origin in the graphical construction did not ring during the mixing part of
  the round.
\vskip 0.1cm   \begin{enumerate}[(1)]
  \item
If $\xi=1$,  set $\o(t_{n+1})_{x_*}=\o(s_n)_{x_*}$ and
    similarly for $\o'$. The remaining part of the configurations at
    time $t_{n+1}$ is sampled using the basic coupling to the left of
    $x_*$ and the maximal coupling for the East process in the
    interval $[x_*+1,-1]$ with boundary condition at $x_*$ equal to $\o(s_n)_{x_*}$.
\vskip 0.1cm
\item If $\xi=0$ we let evolve $\left(\o(s_n),\o'(s_n)\right)$ with the basic
  coupling conditioned to have at least one ring either at $x_*$ or at
  the origin or both.
  \end{enumerate}
  \end{enumerate}
  \end{enumerate}
The final coupling $M^t_{\o,\o'}$ will be obtained by first sampling $\left(\o(t_N),\o'(t_N)\right)$
from $M^{(N)}$ and then by applying the basic coupling for the time
lag $(t-t_N)$.
\end{definition}
It is easy to check that $\{M^{(n)}\}$ is indeed a family of couplings
for $\{\left(\mu_\o^{t_n},\mu^{t_n}_{\o'}\right)\}_{n=0}^N$. Define now
\[
p_n:= M^{(n)}\left(\o\neq \o' \text{ in the interval } I_n\right)
\]
and recall that $\eps$ is the exponent entering in the definition of
the round length $\D$.
\begin{claim}\label{cm1} There exist $\eps_0>0$ such that, for all
  $\eps<\eps_0$ and all $t$ large enough depending on $\eps$,
\[
p_N= O(e^{-t^\a}),
\]
for some positive $\a=\a(\eps)$.
\end{claim}
 \begin{proof}%[Proof of the Claim]
The claim follows from the recursive inequality:
\begin{equation}
  \label{eq:8}
p_{n+1}\le    Ce^{-t^{\eps/2}} +p_n(1-e^{-2\D_2}/2),
\end{equation}
for some constant $C$.
In fact, if we assume \eqref{eq:8} and recall
that $e^{-2\D_2}= t^{-2\kappa  \eps}$, we get
\begin{align*}
p_N & \le C e^{-t^{\eps/2}}[1+(1-e^{-2\D_2}/2)+(1-e^{-2\D_2}/2)^2+\ldots] + \left(1-e^{-2\D_2}/2\right)^N \\
 & \le  2C e^{-t^{\eps/2}}t^{2\kappa  \eps} + \left(1-e^{-2\D_2}/2\right)^N=O(e^{-t^{\eps/3}}),
\end{align*}
provided that $1-\eps(1+ 2\kappa)> \eps/3$, \ie $\eps< 3/(4+6\kappa),$ since $N>c t^{1-\eps}$ for some constant $c$.
Notice crucially  that since $\kappa$ was bounded as $p \downarrow 0$ in the statement of Theorem \ref{th:key3}, $\eps_0$ and $\alpha(\eps)$ can be chosen uniformly as $p \downarrow 0.$\\

\noindent
To prove \eqref{eq:8} we use Lemma~\ref{finitespeed} together with Theorem~\ref{th:key3}.
We begin by examining the possible occurrence of two very unlikely
events each of which will contribute to
the constant term in \eqref{eq:8}.
\begin{itemize}[$\bullet$]
\item  The first possibility is that $\o(t_{n})= \o'(t_{n}) \text{ in the
    interval  } I_n$ and $F(a_{n},a_{n+1};t_n,t_{n+1})$
   occurred. Here $a_n=-v_{\rm
  min}t_n+2v_{\rm max}\D n$ is the left boundary of $I_n$
and similarly for $a_{n+1}$. The linking event could in fact move
possible discrepancies between $\o(t_n),\o'(t_n)$ sitting outside $I_n$ to the inside of $I_{n+1}$. Since $|a_{n}-a_{n+1}|\ge v_{\rm
  max}(t_{n+1}-t_n)$, Lemma~\ref{finitespeed} shows that this case gives a
contribution to $p_{n+1}$ which is $O(e^{-|a_{n}-a_{n+1}|})=O(e^{-v_{\rm max}t^\e})$.
\item The second possibility  is that either $\o(t_n)$ or $\o'(t_n)$ do not
    satisfy the $(\d,\eps)$-weak
spacing condition in $[-v_{\rm min}t_n,-t^\eps_n]$. The bound
\eqref{1} of Theorem
\ref{th:key3} shows that the contribution of such a case is
$O(e^{-t^{\eps/2}})$.
\end{itemize}
Having discarded the occurrence of the above ``extremal'' situations, we now assume that $(\o(t_{n}),\o'(t_{n}))$ are such that: (i) they
are different in the interval
  $I_n$; (ii) they satisfy the $(\d,\eps)$-weak
spacing condition in $[-v_{\rm min}t_n,-t^\eps_n]$. It will be useful
to denote by $\cG_n$ the
set of pairs $(\o,\tilde \o)$ fulfilling (i) and (ii) above.

We will
show that, \emph{uniformly} in $(\o,\tilde\o)\in \cG_n$, the
probability that at the end of the round $\left(\o(\D),\tilde \o(\D)\right)$ are not
coupled inside the interval $I_{n+1}$  is smaller
than $(1- \frac 12 e^{-2\D_2})$. That clearly
proves the second term in \eqref{eq:8}.

To prove that, recall the definition of the $\L_n$-maximal coupling
$MC_{\L_n}$, fix $(\o,\tilde\o)\in \cG_n$ and consider the event $\cB$
that:
\begin{enumerate}[(i)]
\item at the
end of the \emph{burn-in} part of the round $\o(\D_1)=\tilde\o(\D_1)$ in
$\L_n$,
\item the vertex $x_*$ appearing in (ii) of step (b) of
Definition~\ref{coupling} is  within $\eps \log t$ from the right
boundary of $\L_n$ and $\o(\D_1)_{x_*}=\tilde\o(\D_1)_{x_*}=0$,
\item
$\o(\D_1)$ and $\tilde \o(\D_1)$ satisfy SSC  in the interval $[-3(v_{\rm max}/v_{\rm min})\kappa
t^\eps,-\kappa \eps \log t]$.

\end{enumerate}

Theorem~\ref{th:key3} proves that, uniformly in $\o,\tilde\o\in \cG_n$,
\[
MC_{\L_n}(\cB) \ge 1- O(e^{-t^{\eps/2}}) - O(t^{-7\eps}) -
p^{\eps \log t}= 1- O(p^{\eps \log t}).
\]
The first error term takes into account the variation distance
from $\pi$ of the marginals in $\L_n$ of $\bbP_\o^{\D_1}$ and $\bbP_{\tilde
  \o}^{\D_1}$, the second error term bounds the probability that
either $\o(\D_1)$ or $\tilde\o(\D_1)$ do not satisfy the SSC
condition in the interval $[-3(v_{\rm max}/v_{\rm min})\kappa
t^\eps,-\kappa \eps \log t]$ and the third term bounds the $\pi$-probability that the event in item (ii) does not occur.
%is the probability under $\pi$ that the vertex $x_*$ is not within  $\eps \log t$ from the right boundary of $\L_n$.

Next we claim that, for any $\kappa$ large enough and any $z\in \L_n$ at
distance at most $\eps\log t$ from the right boundary of $\L_n$,
\begin{gather}
\sup_{\o,\tilde\o\in \cG_n}\bbP\left(\o(\D)\neq \tilde \o(\D) \text{ in } I_{n+1}\mid
\cB,\{x_*=z\}, \{\xi =1\}\right)\nonumber \\
\label{eq:10}
\le e^{-|a_n-a_{n+1}|}\  + 3\kappa t^\eps
\left(\frac{c^*}{ q}\right)^{\eps \log t}\, e^{-\D_2(\gap(\cL)\wedge m)}= O(t^{-2\eps}).
\end{gather}
The first term in the r.h.s.\ is the probability that the linking
event $F(a_n,a_{n+1};\D_1,\D)$ occurred. The second term comes
from Proposition~\ref{prop:key1} and it bounds from above the probability that, under the
maximal coupling for the East process in the interval $[x_*+1,-1]$ and
in a time lag $\D_2$, we see a discrepancy.

In conclusion, the probability that $\o(\D)= \tilde \o(\D)$ in
$I_{n+1}$ is larger
than
\[
MC_{\L_n}(\cB)(1-o(1))\bbP(\xi=1)\ge  \frac 12 e^{-2\D_2},
\]
thus proving the claim.
\end{proof}
We are now in a position to finish the proof of Theorem~\ref{coupling}. Let
$v^*\equiv v_{\rm min} -3\eps v_{\rm max}$ and let $a_N=-v_{\rm min}t_N+\eps v_{\rm max} t$ be the left boundary of the interval $I_N=[a_N,0]$.
Since by Remark~\ref{speedbnd}, $v_{\rm min}$ converges to $1$ as $p \downarrow 0,$ $v^*$ can be chosen uniformly as $p \downarrow 0.$\\

\noindent
Pick two configurations
$\o(t_N),\o'(t_N)$ at time $t_N$ and make them evolve under the basic
coupling until time $t$. Clearly the events
$\{\o(t_N)_x=\o'(t_N)_x\ \forall x\in I_N\}$ and $\{\exists\, x\in [-v^*t,0]:\ \o_x(t)\neq
\o'_x(t)\}$ imply  the linking event
$F(a_N,-v^*t;t_N,t)$ from Lemma~\ref{finitespeed}. By construction
$|v^*t-a_N|=\eps v_{\rm max}t \ge v_{\rm max}(t-t_N)$ for large enough $t$. Therefore,
\begin{align*}
 M^t_{\o,\o'}(\exists\, x\in [-v^*t,0]:\ \o_x\neq \o'_x)&\le p_N +
 \bbP\left(F(a_N,-v^*t;t_N,t)\right)\\
&\le O(e^{-t^\a})+ e^{-\eps v_{\rm max}t},
\end{align*}
as required. Moreover, by the proof of Claim \ref{cm1}, $\alpha$ can be chosen uniformly as $p \downarrow 0$. Thus we are done.\qed

\subsection{Mixing properties of the front increments: Proof of Corollary~\ref{cor:wf}}
To prove \eqref{eq:20} we observe that, for any $n\ge v_{\rm max}\D$, the event $|\xi_1|\ge n$
implies the occurrence of the linking event $F(0,n;0,\D)$. Lemma~\ref{finitespeed} now gives that
\begin{equation}
  \label{eq:9}
\bbE_\o\left[f(\xi_1)^2\right] \le \max_{|x|\le v_{\rm max}\D}f(x)^2
+\sum_{n\ge v_{\rm max}\D} f(n+1)^2 e^{-n} <\infty.
\end{equation}
In order to prove \eqref{eq:20tris} we apply the Markov property at
time $t_{n-1}$ and write
\begin{gather*}
\bbE_\o\left[f(\xi_n)\right]=
\int d\mu^{t_{n-1}}_\o(\o')\, \bbE_{\o'}\left[f(\xi_1)\right].
\end{gather*}
At this stage we would like to appeal to Theorem
\ref{coupling} to get the
 sought statement. However Theorem
\ref{coupling} only says that, for any $t$ large enough, $\mu^t_\o$ is very close to the
invariant measure $\nu$ in the interval $[-v^*t,0]$. In order to
overcome this problem,
for any $\o\in \O_\Front$ and any $t>0$ we define $\Phi_t(\o)\in \O_\Front$
as that configuration which is equal to $\o$ in $[-v^* t,0]$ and identically  equal to $1$
elsewhere. Then, under the basic coupling, the front at time $t$
starting from $\Phi_t(\o)$ is different from the front starting from
$\o$
iff the linking event $F(-v^*t,0;0,\D)$ occurred.

In conclusion,  if $v^* t_{n-1}\ge v_{\rm max}\D$,
\begin{align*}
\sup_{\o\in \O_\Front}&\bigg|\int d\mu^{t_{n-1}}_\o(\o')\,
\bbE_{\o'}\left[f(\xi_1)\right] -
\int d\mu^{t_{n-1}}_\o(\o')\, \bbE_{\Phi_{t_{n-1}}(\o')}\left[f(\xi_1)\right]\bigg|\\
&\le
\bbP(F(-v^*t_{n-1},0;0,\D))^{1/2}\sup_{\o\in
  \O_\Front}\bbE_\o\left[f(\xi_1)^2\right]^{1/2}\\
&\le
e^{-v^*t_{n-1}/2} \sup_{\o\in
  \O_\Front}\bbE_\o\left[f(\xi_1)^2\right]^{1/2}.
\end{align*}
We can now apply Theorem~\ref{coupling} to get that
\begin{gather*}
\bigg|\int d\mu^{t_{n-1}}_\o(\o')\,
\bbE_{\Phi_{t_{n-1}}(\o')}\left[f(\xi_1)\right]-\bbE_{\nu}\left[f(\xi_1)\right]\bigg|
\\
\le
\left[ \sup_{\o\in \O_\Front}\|\mu_\o^{t_{n-1}}-\nu\|^{1/2}_{[-v^*t_{n-1},0]}    +
e^{-v^*t_{n-1}/2}\right]\sup_{\o\in
  \O_\Front}\bbE_\o\left[f(\xi_1)^2\right]^{1/2}
= O(e^{-t_{n-1}^\a/2}).
\end{gather*}
To prove \eqref{eq:13} suppose first that $v^*(j-1)\ge v_{\rm max}(n-j)$
where $v^*$ is the constant appearing in Theorem~\ref{coupling}. Then we can
use the Markov property at time $t_{j-1}$ and repeat the previous
steps to get the result. If instead $v^*(j-1)\le
v_{\rm max}(n-j)$ it suffices to write
\[
{\cov}_\o\left(\xi_j,\xi_n \right)= {\cov}_\o\left(\xi_j,\bbE_\o[\xi_n\tc \cF_{t_{j}}] \right)
\]
and apply \eqref{eq:20tris} to $\bbE_\o[\xi_n\tc \cF_{t_{j}}]$ to
get that in this case
\begin{equation}
  \label{eq:16}
 \sup_{\o\in \O_\Front}\left| {\cov}_\o\left(\xi_j,\xi_n \right) \right| =O(e^{-\g (n-j)^\a})
\end{equation}
for some constant $\g$ depending on $v^*,v_{\rm max}$.
Following the exact steps as above after replacing $\xi_j,\xi_n$ by $f(\xi_j),f(\xi_n)$ yields \eqref{eq:13fabio}.
Finally, \eqref{eq:21} follows from exactly the same steps leading to the proof of \eqref{eq:20tris}.
\qed

\section{Proofs of main results}\label{sec:main-proofs}
\subsection{Proof of Theorem~\ref{th:main1}}
We begin with the proofs of \eqref{th1.1} and \eqref{th1.2}.

As far as \eqref{th1.2} is concerned, this follows directly from observing that
\[
\frac{d}{dt} \bbE_{\o}\left[X(\o(t))\right]= q-p\mu^t_\o(\o(-1)=0)=v +O(e^{-t^\a}).
\]
Appealing to \eqref{eq:19} and Corollary~\ref{cor:wf} we get immediately that for any $\o\in \O_\Front$
\[
\bbE_\o\left[\left((X(\o(t))-vt)/t\right)^4\right]=O(t^{-2})
\]
and \eqref{th1.1} follows at once.

\medskip

We next prove \eqref{th1.3}. Using Corollary~\ref{cor:wf} with $f(x)=x^2$, we get that, for any $n$ large enough,
\begin{align*}
\var_\o(\xi_n)= \var_\nu(\xi_1)
+ O(e^{-n^\a}).
\end{align*}
Hence
\begin{gather*}
\lim_{t\to \infty} \frac 1t \left[\sum_{n=1}^{N_t} \var_\o(\xi_n) +
  \var_\o(X(\o(t))-X(\o(t_N)))\right]
=\D^{-1}\var_\nu(\xi_1).
\end{gather*}
Moreover, \eqref{eq:13} implies that
\begin{gather*}
\lim_{t\to \infty} \frac 2t \Bigg[\sum_{j<n}^{N_t} {\cov}_\o(\xi_j,\xi_n) +
 \sum_{n=1}^{N_t} {\cov}_\o(\xi_n,X(\o(t))-X(\o(t_N)))\Bigg]
\\=\frac{2}{\D}\sum_{n\ge 2}{\cov}_\nu\left(\xi_1,\xi_{n}\right),
\end{gather*}
the series being absolutely convergent because of \eqref{eq:16}.
In conclusion, for any $\o\in \O_\Front$
\begin{gather}
\label{eq:18}
\lim_{t\to \infty}
\frac 1t \var_\o\left(X(\o(t))\right)=\D^{-1}\Bigg[\var_\nu(\xi_1)+2\sum_{n\ge 2}{\cov}_\nu\left(\xi_1,\xi_{n}\right)\Bigg].
\end{gather}
Next we show that for $p$ small enough the r.h.s.\ of \eqref{eq:18} is positive.
We first observe that there exists $c=c(p)$ such that $\limsup_{p\to
  0^+}c(p)<\infty$ and
\begin{equation}
  \label{eq:facile}
\sup_{\D}\sum_{n\ge 2}|{\cov}_\nu\left(\xi_1,\xi_{n}\right)|\le c(p).
\end{equation}
To prove \eqref{eq:facile} assume without loss of generality that
$\D\in \bbN$ and write $\xi_1= \sum_{i=1}^\D \xi_i'$
and $\xi_{n}=\sum_{i=(n-1)\D+1}^{n\D} \xi_i'$, where the increments $\xi_i'\,$'s
refer to a unit time lag. Thus
\[
\sum_{n\ge 2}|{\cov}_\nu\left(\xi_1,\xi_{n}\right)| \le
\sum_{n\ge 2}\sum_{i=1}^\D\sum_{j=(n-1)\D+1}^{n\D}|{\cov}_\nu\left(\xi'_i,\xi'_{j}\right)|
\]
The claim now follows from \eqref{eq:13} together with the fact that
the constants $\a,v^*$ are uniformly bounded away from zero as $p\to
0$.

Thus, in order to show
that the r.h.s.\ of \eqref{eq:18} is positive, it is enough to show that
it is possible to choose $\D$ and $p$ such that $
\var_\nu\left(\xi_1\right) > \limsup c(p)$.

Recall that $q^*=\nu(\o_{-1}=0)$. Then a little computation shows that
\begin{align}
\frac{d}{dt} \var_\nu\left(X(\o(t))\right)&=q+pq^*-2p\,{\cov}_\nu\left(X(\o(t)),{\mathds
    1}_{\{\o(t)\in \O^{**}\}}\right)\nonumber\\
&\ge q+pq^*-2p
\var_\nu\left(X(\o(t))\right)^{1/2}\left(q^*(1-q^*)\right)^{1/2}\\
\label{eq:17}&\ge q+pq^*-p\var_\nu\left(X(\o(t))\right)^{1/2},
\end{align}
where $\O^{**}=\{\o\in \O^*:\ \o_{X(\o)-1}=0\}$.

If
$\left[\var_\nu\left(\xi_1\right)\right]^{1/2}\le  \frac{q+pq^*}{2p}$
for all $\D>0$, then \eqref{eq:17} implies that
\[
\lim_{\D\to \infty}\var_\nu\left(\xi_1\right)=\infty.
\]
Otherwise there exists $\D>0$ such that
$
\left[\var_\nu\left(\xi_1\right)\right]^{1/2}\ge
\frac{q+pq^*}{2p}
$;
hence, the desired inequality~\eqref{th1.3} follows by taking $p$ small enough.

It remains to prove~\eqref{th1.4}. If $\s^*=0$, then necessarily
\[
\sup_\D \var_\n\left(\xi_1\right)<\infty.
\]
In this case the Chebyshev
inequality suffices to prove that, for any $\o\in \O_\Front$,
\[
(X(\o(t))-vt)/\sqrt{t}\ \stackrel{\bbP_\o}{\longrightarrow}\ 0, \quad
\text{ as }t\to \infty.
\]
If instead $\s^*>0$, we appeal to an old result on the central limit
theorem for mixing stationary random fields \cite{Bolthausen}. Unfortunately our mixing
result, as expressed e.g.\ in Corollary~\ref{cor:wf} (cf.~\eqref{eq:21}), is not exactly
what is needed there and we have to go through some of the steps of
\cite{Bolthausen} to prove the sought statement.

Consider the sequence $\{\xi_j\}$ defined above (with e.g.\ $\D=1$) and let $\bar \xi_j:= \xi_j- v\D$.  Further let
$S_n=\sum_{j=1}^n \bar \xi_j$.
It suffices to prove that, for all
$\o\in \O_\Front$, the law of $S_n/\s_* \sqrt{n}$ converges to the normal
law $\cN(0,1)$. As in \cite{Bolthausen} let $f_N(x)=\max\left[\min(x,N),-N\right]$ and
let $\tilde f_N(x):=x-f_N(x)$. Clearly $\var(\tilde f_N(\bar \xi_j))\rightarrow 0$ as $N \rightarrow \infty$ uniformly in $j.$

Then Corollary~\ref{cor:wf} \eqref{eq:13fabio}  implies that
\[
\bbE_\o\left[\frac{\sum_{j=1}^n \tilde f_N(\bar \xi_j)-\bbE_\o[\tilde
    f_N(\bar \xi_j)]}{n^{1/2}}\right]^2=\frac 1n \sum_{j,k=1}^n {\cov}_\o\left(\tilde f_N(\bar \xi_j),\tilde
  f_N(\bar \xi_k)\right)
\]
converges to $0$ as $N\to \infty$ uniformly in $n$. Hence it is enough
to prove the result for the truncated variables $f_N(\bar \xi_j)$. For
lightness of notation we assume henceforth that the $\bar\xi_j$'s are bounded.

Let now $\ell_n= n^{1/3}$ and let
\begin{align*}
S_{j,n}= \sum_{k=1}^n{\mathds 1}_{|k-j|\le \ell_n}\bar\xi_k, \quad
\a_n= \sum_{j=1}^n\bbE_\o\left[\bar \xi_jS_{j,n}\right],\quad j\in \{1,\dots,n\}.
\end{align*}
The decay of covariances \eqref{eq:13} implies that $\a_n=
\var_\o(S_n)+ o(1)$. Hence it is enough to show that $S_n/\sqrt{\a_n}$
is asymptotically normal. The main observation of \cite{Bolthausen},
in turn inspired by the Stein method \cite{Stein}, is that the latter property
of $S_n/\sqrt{\a_n}$ follows if
\begin{equation}
  \label{eq:190}
\lim_{n\to \infty} \bbE_\o\left[(i\l-S_n)e^{i\l
    \frac{S_n}{\sqrt{\a_n}}}\right]=0, \quad \forall \l\in \bbR.
\end{equation}
In turn \eqref{eq:190} follows if (see \cite{Bolthausen}*{Eqs.~(4)--(5)})
\begin{align}
  \label{eq:19bis}
  \lim_{n\to \infty} \bbE_\o\Bigl[\bigl(1-\frac{1}{\sqrt{\a_n}}\sum_{j=1}^n\bar\xi_j
    S_{j,n}\bigr)^2\Bigr]&=0\\
    \label{eq:19tris}
\lim_{n\to \infty} \frac{1}{\sqrt{\a_n}} \bbE_\o\Bigl[\Big|\ \sum_{j=1}^n \bar\xi_j\bigl(1-e^{-i\l
      \frac{S_n}{\sqrt{\a_n}}}-i\l S_{j,n}\bigr)\Big|\Bigr]&=0\\
  \label{eq:19quatris}\lim_{n\to \infty}
   \frac{1}{\sqrt{\a_n}}\sum_{j=1}^n\bbE_\o\Bigl[\bar \xi_j\, e^{i\l
       \frac{(S_n-S_{j,n})}{\sqrt{\a_n}}}\Bigr] &=0.
\end{align}
 As in \cite{Bolthausen}, the mixing properties \eqref{eq:13} and
 \eqref{eq:21} easily prove that \eqref{eq:19bis} and
 \eqref{eq:19tris} hold.
 As far as \eqref{eq:19quatris} is concerned the formulation of Theorem~\ref{coupling} forces us to argue a bit differently than
\cite{Bolthausen}. We first observe that, using the boundedness of the variables $\bar \xi_j$'s, \eqref{eq:19quatris}
is equivalent to
 \begin{equation}
   \label{eq:20bis}
   \lim_{n\to \infty}
   \frac{1}{\sqrt{\a_n}}\sum_{j=\ell_n}^n\bbE_\o\left[\bar \xi_j\, e^{i\l
       \frac{(S_n-S_{j,n})}{\sqrt{\a_n}}}\right]=0,
\quad \forall \l\in \bbR.
 \end{equation}
Fix two numbers $M$ and $L$ with $L\le M/10$ (eventually they will be
chosen logarithmically increasing in $n$)
 and write
\begin{align*}
e^{i\l \frac{(S_n-S_{j,n})}{\sqrt{\a_n}}}&= \sum_{m=0}^M
\frac{(i\l)^m}{m!}\left(\frac{(S_n-S_{j,n})}{\sqrt{\a_n}}\right)^m \\ &+ \sum_{m=M+1}^\infty
\frac{(i\l)^m}{m!}\left(\frac{(S_n-S_{j,n})}{\sqrt{\a_n}}\right)^m{\mathds
1}_{\{| \frac{(S_n-S_{j,n})}{\sqrt{\a_n}}|\le L\}}
\\ &+ \left[e^{i\l \frac{(S_n-S_{j,n})}{\sqrt{\a_n}}}- \sum_{m=0}^M
\frac{(i\l)^m}{m!}\left(\frac{(S_n-S_{j,n})}{\sqrt{\a_n}}\right)^m\right]{\mathds
1}_{\{| \frac{(S_n-S_{j,n})}{\sqrt{\a_n}}|> L\}}\\
&=: \ Y^{(j)}_1+Y^{(j)}_2+Y^{(j)}_3.
\end{align*}
Let us first examine the contribution of $Y^{(j)}_2$ and $Y^{(j)}_3$ to
the covariance term \eqref{eq:20bis}.
Using the boundedness of the variables $\{\bar \xi_j\}_{j=1}^n$ there exists a positive constant $c$ such that:
\begin{align*}
\frac{1}{\sqrt{\a_n}}\sum_{j=\ell_n}^n|\bbE_\o\left[\bar
  \xi_j\,Y^{(j)}_2\right]| &\le c \,\sqrt{n} \ \frac{L^{M+1}}{M!},\\
\frac{1}{\sqrt{\a_n}}\sum_{j=\ell_n}^n|\bbE_\o\left[\bar
  \xi_j\,Y^{(j)}_3\right]|&\le c\,\sqrt{n}\max_j \bbE_\o\left[e^{2|\l|
    \frac{|S_n-S_{j,n}|}{\sqrt{\a_n}}}\right]^{1/2}
\bbP_\o\left(| \frac{(S_n-S_{j,n})}{\sqrt{\a_n}}|> L\right).
\end{align*}
\begin{lemma}
\label{large-dev} There exists $c>0$ such that,
for all $n$ large enough and any $\b=O(\log n)$,
  \begin{equation}
    \label{eq:22}
    \bbE_\o\left[e^{\b\frac{|S_n-S_{j,n}|}{\sqrt{\a_n}}}\right] \le
   2 e^{c \b^2}.
  \end{equation}
Moreover, there exists $c'>0$ such that, for all $n$ large enough and
all $L\le \log n$,
\begin{equation}
  \label{eq:23}
  \bbP_\o\left(| \frac{(S_n-S_{j,n})}{\sqrt{\a_n}}|> L\right) \le
  e^{-c' L^2}.
\end{equation}
\end{lemma}
Assume for the moment the lemma and choose
$L=M/10$ and $M=\log n$. We can conclude that
\begin{gather*}
\frac{1}{\sqrt{\a_n}}\sum_{j=\ell_n}^n|\bbE_\o\left[\bar
  \xi_j\,(Y^{(j)}_2+Y^{(j)}_3)\right]|
\le
C\sqrt{n}\left[e^{-c'L^2}+ \frac{L^{M+1}}{M!}\right],
\end{gather*}
so that
\[
\lim_{n\to \infty}\frac{1}{\sqrt{\a_n}}\sum_{j=\ell_n}^n|\bbE_\o\left[\bar
  \xi_j\,(Y^{(j)}_2+Y^{(j)}_3)\right]|=0.
\]
We now examine the contribution of $Y^{(j)}_1$ to \eqref{eq:20bis}. Recall $$S_n-S_{j,n}=\sum_{\substack{1\le i \le n \\ |i-j|>\ell_n}} \bar
\xi_i.$$ Thus clearly,
\[
\frac{1}{\sqrt{\a_n}}\sum_{j=\ell_n}^n\bbE_\o\left[\bar
  \xi_j\,(Y^{(j)}_1)\right]=\frac{1}{\sqrt{\a_n}}\sum_{j=\ell_n}^n\sum_{m=1}^M
\left( \frac{i\l}{\sqrt{n}}\right)^m\sumtwo{i_1,\dots, i_m}{\min_k
   |i_k-j|\ge \ell_n}\bbE_\o\left[\bar \xi_j \prod_{i=k}^m \bar \xi_{i_k}\right],
 \]
where the labels $i_1,\dots i_m$ run in $\{1,2,\dots, n\}$.
\begin{lemma}
\label{covar2}
Let $M=\log n$. Then, for any $m\le M$, any $j\in\{\ell_n,\dots,n\}$ and any $\{i_1,\dots i_m\}$ satisfying $\min_k
   |i_k-j|\ge \ell_n\,$, it holds that
\[
|\bbE_\o\Bigl[\bar \xi_j \prod_{i=k}^m \bar \xi_{i_k}\Bigr]|=O(e^{-n^{\a/6}}).
\]
Here $\a$ is the mixing exponent appearing in Theorem~\ref{coupling}.
\end{lemma}
Assuming the lemma we get immediately that also
\[
\lim_{n\to \infty}\frac{1}{\sqrt{\a_n}}\sum_{j=\ell_n}^n \bbE_\o\left[\bar
  \xi_j\,(Y^{(j)}_1)\right]=0
\]
and \eqref{eq:20bis} is established. In conclusion, \eqref{th1.4} would follow from Lemmas~\ref{large-dev}--\ref{covar2}.

\begin{proof}[Proof of Lemma~\ref{large-dev}]
Let us begin with \eqref{eq:22}. For simplicity we prove that,
for any constant $\b=O(\log n)$, $\bbE_\o\left[\exp(\b
  S_n/\sqrt{n})\right]\le e^{c\b^2}$ for some constant $c>0$.
Similarly one could proceed for $\bbE_\o\left[\exp(-\b
  S_n/\sqrt{n})\right]$ and get that
\[
\bbE_\o\left[\exp(\b |S_n|/\sqrt{n})\right]\le \bbE_\o\left[\exp(\b
  S_n/\sqrt{n})\right] +
\bbE_\o\left[\exp(-\b S_n/\sqrt{n})\right]\le 2 e^{c\b^2}.
\]
We partition the discrete interval $\{1,2,\dots, n\}$ into disjoints blocks of
cardinality $n^{1/3}$. Given a integer $\kappa$, by applying the Cauchy-Schwarz inequality a
finite number of times depending on $\kappa$, it is sufficient to prove the result for
$S_n$ replaced by the sum $S^{(\kappa)}_{\cB}$ of the $\bar\xi_j$'s
restricted to an arbitrary collection $\cB$ of blocks with the property that
any two blocks in $\cB$ are
separated by at least $\kappa$ blocks.

Fix one such collection $\cB$ and let $B$ be the rightmost block
in $\cB$. Let $n_B$ be the largest label in $\cB$ which is not in the
block $B$ and let
$t_B=n_B\D$ be the corresponding time. Further let $Z_B=\sum_{j\in
  B}\bar\xi_j$. If $c\kappa>v_{\rm max}$ where $c$ is the constant
appearing in Theorem~\ref{coupling}, we can appeal to \eqref{eq:21} to
obtain
\[
\bbE_\o\left[\exp(\b Z_B/\sqrt{n})\tc \cF_{t_B}\right]= \bbE_\nu\left[\exp(\b
  Z_B/\sqrt{n})\right]+ O(e^{-n^{\a/3}}e^{\b n^{-1/6}}).
\]
Using the trivial bound
$Z_B/\sqrt{n}= O(n^{-2/3})$ we have
\[
\bbE_\nu\left[\exp(c Z_B/\sqrt{n})\right] = 1 +\frac{\b^2}{2n}\var_\nu(Z_B) + O(\b^3
n^{-7/6})\var_\nu(Z_B),
\]
where $\var_\nu(Z_B)= O(n^{1/3})$ thanks to \eqref{eq:13}. Above we used
the trivial bound
\[
\bbE_\nu\left[|Z_B|^3\right]\le c\, n^{1/3}\var_\nu(Z_B).
\]
In conclusion, using the apriori bound $\b\le \log n$, we get that
\[
 \bbE_\o\left[\exp(\b Z_B/\sqrt{n})\tc \cF_{t_B}\right] \le 1+c \frac{\b^2}{n^{2/3}}.
\]
The Markov property and a simple iteration imply that,
\[
\bbE_\nu\left[\exp\big(\b S^{(\kappa)}_{\cB}/\sqrt{n}\big)\right]\le \left[1+c
  \frac{\b^2}{n^{2/3}}\right]^{|\cB|}\le \exp(c' \b^2),
\]
uniformly in the cardinality $|\cB|$ of the collection. The bound
\eqref{eq:22} is proved.

The bound \eqref{eq:23} follows at once from \eqref{eq:22} and
the exponential Chebyshev inequality
\[
  \bbP_\o\left(| \frac{(S_n-S_{j,n})}{\sqrt{\a_n}}|> L\right) \le
e^{-\b L}\,\bbE_\o\left[\exp(| \frac{(S_n-S_{j,n})}{\sqrt{\a_n}}|)\right],
\]
with $\b=\e L$,  $\e$ being a
sufficiently small constant.
\end{proof}
\begin{proof}[Proof of Lemma~\ref{covar2}] Fix $j\in [1,\dots,n]$ and $m\le \log n$, together with
a choice of labels
$1\le i_1\le i_2 \le\dots \le i_m\le n$ such that $\min_k|i_k-j|\ge
\ell_n$. Let $t_{i_k}=i_k\D$. \\
$\bullet$ If $i_m\le j-\ell_n$ then we can apply the Markov property at time
$t_{i_m}$
together with Corollary~\ref{cor:wf} to get
\[
\bigg|\bbE_\o\Bigl[\bar \xi_j \prod_{i=k}^m \bar \xi_{i_k}\Bigr]\bigg|\le
e^{-n^{\a/3}}\bbE_\o\Bigl[\prod_{i=k}^m |\bar \xi_{i_k}|\Bigr]\le
c^{m}e^{-n^{\a/3}}.
\]
$\bullet$ If instead there exists $b\le m-1$ such that $i_b < j < i_{b+1}$
we need to distinguish between two sub-cases.

(a) For all $k\ge b+2$ it holds that
$i_k-i_{k-1}\le n^{1/6}$ and in particular
$t_m-t_{b+1}\le n^{1/6}\D$. In this case the fact that
$t_{b+1}-j\D \ge \ell_n$ and
$v_{\rm max}(t_m-t_{b+1})\ll \ell_n$ together with \eqref{eq:21}, imply that
\[
\bbE_\o\Bigl[\bar \xi_j \prod_{i=k}^m \bar \xi_{i_k}\Bigr]=
\bbE_\o\Bigl[\bar \xi_j \prod_{i=k}^{b} \bar
  \xi_{i_k}\Bigr]\left[\bbE_\nu\Bigl[\prod_{i=b+1}^{m} \bar
    \xi_{i_k}\Bigr]+
O\Bigl(e^{-n^{\a/3}}{\rm poly}(n)\Bigr)\right].
\]
The conclusion of the lemma then follows from the previous case
$i_m\le j-\ell_n$.

(b) We now assume that $k^*:=\max\{k\ge b+1:\ i_{k+1}\ge
  i_{k}+n^{1/6}\}<n$. By repeating the previous step with the
  Markov property applied at time
  $t_{i_{k^*}}$ we get
 \[
\bbE_\o\Bigl[\bar \xi_j \prod_{i=k}^m \bar \xi_{i_k}\Bigr]=
\bbE_\o\Bigl[\bar \xi_j \prod_{i=k}^{k^*} \bar
  \xi_{i_k}\Bigr]\left(\bbE_\nu\Bigl[\prod_{i=k^*+1}^{m} \bar
    \xi_{i_k}\Bigr]+
O\Bigl(e^{-n^{\a/3}}{\rm poly}(n)\Bigr)\right).
\]
By iterating the above procedure we can reduce ourselves to case (a)
and get the sought result.
\end{proof}
As Lemmas~\ref{large-dev}--\ref{covar2} imply~\eqref{th1.4}, this concludes the proof of Theorem~\ref{th:main1}.
\qed

\begin{remark}\label{rem:sigma*-small-p}
The above proof also established that the limiting variance
$\s_*^2$ is strictly positive for all $p$ small enough.
\end{remark}

\subsection{Proof of Theorem~\ref{th:main2}}
\label{East-cutoff}
Given the interval $\L=[1,\dots,L]$ and $\o\in \O_\L$, let $\bbP^{\L,t}_{\o}$ be the law of
the process started from $\o$. Recall that
\[
\|\bbP^{\L,t}_{\o}-\bbP^{\L,t}_{\o'}\|=\inf\{M(\o(t)\neq \o'(t)):\ M \text{ a coupling of } \bbP^{\L,t}_{\o} \text{ and } \bbP^{\L,t}_{\o'}\},
\]
and introduce the hitting time
\[
\tau(L)=\inf\{t: X(\o(t))=L\},
\]
where the initial configuration is identically equal to one (in the
sequel $\mathbf 1$). It
is easy to check (see, e.g.,~\cite{East-survey}) that at time $\t(L)$
the basic coupling (cf.~\S\ref{setting-notation}) has coupled all initial
configurations.
Thus
\[
d_{\TV}(t)\le \sup_{\o,\o'}\|\bbP^{\L,t}_{\o}-\bbP^{\L,t}_{\o'}\|\le
\bbP^{\L}(\tau(L)> t).
\]
Using the graphical construction, up to time $\t(L)$  the East process in $\L$ started from the configuration
$\mathbf{1}$ coincides with the infinite East process started from the
configuration $\o^*\in \O_\Front$ with a single zero at the
origin. Therefore
\[
\bbP^{\L}_{\mathbf{1}}(\tau(L)> t) \le \bbP_{\o^*}(X(\o(t))<L),
\]
thus establishing a bridge with Theorem~\ref{th:main1}. Recall now the
definition of $\s_*$ from Theorem~\ref{th:main1} and distinguish
between the two cases $\s_*>0$ and $\s_*=0$.

\smallskip\noindent$\bullet$
The case $\s_*>0$. Here we will show that
\begin{equation}
  \label{eq-tmix(epsilon)-sigma*>0}
  \tmix(L,\epsilon)= v^{-1}L  + (1+o(1))\frac{\sigma_{*}}{v^{3/2}}\Phi^{-1}(1-\eps)\, \sqrt{L}\,.
\end{equation}
For $s\in \bbR$, let $t_\star=L/v+s\sqrt{L}$. Then \eqref{th1.3} implies that
\[
\bbP_{\o^*}\left(X(\o({t_\star}))<L\right)= \bbP_{\o^*}\Bigl(
\frac{X(\o(t_\star))-vt_\star}{\sqrt{L/v}}< -v^{3/2} s\Bigr)\rightarrow
\Phi\Bigl(-\frac{v^{3/2}s}{\sigma_{*}}\Bigr)
\]
as $L\to \infty.$
Hence,
\begin{equation}\label{sup1}
\limsup_{L\to \infty} d_{\TV}(L/v+s\sqrt{L}) \le \Phi\Bigl(-\frac{v^{3/2}s}{\sigma_{*}}\Bigr).
\end{equation}
To prove a lower bound on the total variation norm, set $a_L=\log L$ (any diverging sequence which is $o(\sqrt{L})$ would do here)
 and define the event $$A_t=\bigl(\o_x(t)=1 \text{ for all }x\in{(L-a_L,L]}\bigr).$$
Then
$$
\bbP^\L_{\mathbf{1}}(A_t)\ge \bbP_{\o^*}\left(X(\o(t))\le L-a_L\right) \quad\text{and}\quad
\pi(A_t)=p^{a_L} = o(1),
$$
and so any lower bound on $\bbP_{\o^*}(X(\o(t_\star))\le L-a_L)$ would translate to a lower bound on $d_\TV(t_\star)$ up to an additive $o(1)$-term.
Again by \eqref{th1.3},
\[ \bbP_{\o^*}\Bigl(X(\o(t_\star))\le L-a_L\Bigr)=\bbP_{\o^*}\biggl(
\frac{X(\o(t_\star))- vt_\star}{\sqrt{L/v}}\le -v^{3/2}s -a_L\sqrt{v/L}\  \biggr)\rightarrow \Phi\Bigl(-\frac{v^{3/2}s}{\sigma_{*}}\Bigr)\]
as $L\to \infty.$
Thus we conclude that
\begin{equation}\label{inf1}
\liminf_{L\to \infty} d_{\TV}(L/v+s\sqrt{L}) \ge \Phi\Bigl(-\frac{v^{3/2}s}{\sigma_{*}}\Bigr).
\end{equation}
Eq.~\eqref{eq-tmix(epsilon)-sigma*>0} now follows from \eqref{sup1} and \eqref{inf1} by choosing $s=\s_* v^{-3/2}\Phi^{-1}(1-\e).$

\smallskip\noindent$\bullet$
The case $\s_*=0.$  Here a similar argument shows that
\[ \tmix(L,\epsilon)= v^{-1} L + O_\epsilon(1),\]
using the fact (following the results in \S\ref{sec:front}) that
$\sup_{\o}\sup_{t}\var_{\o}(X(\o(t))) < \infty$ if $\s_*=0$.

\smallskip
This concludes the proof of Theorem~\ref{th:main2}.\qed
\section{Cutoff and concentration for constrained models on
  trees}\label{sec:trees}
In this section we consider constrained oriented models on
regular trees and prove strong concentration results for hitting
times which are the direct analog of the hitting time $\t(L)$ define in \S\ref{East-cutoff} for the East process. As a
consequence we derive a strong cutoff result for the ``maximally
constrained model'' (see below).
\subsection{Kinetically constrained models on trees} Let $\bbT$ be the $k$-ary rooted tree, $k\ge 2$, in which each vertex
$x$ has $k$ children. We will denote by $r$ the root and by $\bbT_L$
the subtree of $\bbT$ consisting of the first $L$-levels starting
from the root.

In analogy to the East process, for a given integer $1\le j\le k$
consider the constrained oriented
  process OFA-jf on $\O=\{0,1\}^{\bbT}$ (cf.~\cite{MT}) in which
  each vertex waits an independent mean one exponential
time and then, provided that $j$ among its children are in state
$0$, updates its spin variable $\o_x$ to $1$ with probability $p$
and to $0$ with probability $q=1-p$. It is known that this process exhibits an ergodicity breakdown above a certain critical probability $p=p_c(k,j)$ (defined more precisely later).
In this paper we will only
examine the two extreme cases $j=1$ and $j=k$ which will be referred to
in the sequel as the \emph{minimally} and \emph{maximally} constrained
models.

The finite volume version of the
OFA-jf process is a continuous time Markov chain
on $\O_{\bbT_L}=\{0,1\}^{\bbT_L}$. In this case, in order to guarantee
irreducibility, the variables at leaves of $\bbT_L$ are assumed to be
unconstrained. As in the case of the East process, the product Bernoulli$(p)$
measure $\pi$ is the unique reversible measure and the same graphical
construction described in \S\ref{setting-notation} holds in this new context.
\subsection{New Results} We are now in a position to state our
results for the minimally and maximally constrained finite volume OFA-jf models. Recall that
\[
\tmix(L,\e) := \inf\{t:\ \max_{\o\in \O_{\bbT_L}}\|\mu^t_\o-\pi\|\le
\e\},\quad \e\in (0,1)
\]
and define $T_{\rm hit}(L):=\bbE\left[\t(L)\right]$,
where $\t(L)$ is the first legal ring for the root for the OFA-jf
process on $\O_{\bbT_L}$ started from the configuration identically equal
to one.
Our first result addresses the concentration of $\t(L)$. Recall that $O_\delta(\cdot)$ denotes that the implicit constant may depend on $\delta$.
\begin{theorem}
\label{th:main3}
The following hold for the centered variable
$\t(L)-T_{\rm hit}(L)$, denoted $\bar \t(L)$.
\begin{enumerate}[(i)]
\item Consider either the minimally or the maximally constrained model
  and fix $p<p_c$. For any fixed $\d>0$, if $n\in\bbN$ is large enough there exists $L_n\in [n,(1+\d)n]$ such that
\[
\bbE|\bar \t(L_n)| = O_\delta(1).
\]
\item Consider the maximally constrained model and choose
  $p=p_c$. For
  any fixed $\d>0$, if $n\in\bbN$ is large enough then there exists $L_n\in [n,(1+\d)n]$ such that
\[
\bbE|\bar \t(L_n)| = O_\delta\left(L_n^{-1}\,  T_{\rm hit}(L_n)\right).
\]
\end{enumerate}
\end{theorem}
The second result concerns the cutoff phenomenon.
\begin{theorem}
\label{th:main4} Consider the maximally constrained model.
\begin{enumerate}[(i)]
\item If $p<p_c$ then for any $\d>0$ and any large enough $n$  there exists $L_n\in [n,(1+\d)n]$ such that
\[
|\tmix(L_n,\e)-T_{\rm hit}(L_n)|= O_{\e,\d}(1) \quad \forall \e\in (0,1).
\]
\item If $p=p_c$ then for any $\d>0$ and any large enough $n$ there exists $L_n\in [n,(1+\d)n]$ such that
\[
|\tmix(L_n,\e)-T_{\rm hit}(L_n)|=O_{\e,\d}\bigl(L_n^{-1} \,T_{\rm hit}(L_n)\bigr) \quad\forall \e\in(0,1).
\]
\end{enumerate}
\end{theorem}

\subsection{Previous work}
Before proving our results we recall the main findings of \cite{MT}
and \cite{CMRTtree}. We now formally define the critical density for the OFA-jf model:
\[
p_c =\sup\{p\in[0,1]:\text{0 is simple eigenvalue of } \cL\},
\]
where $\cL$ is the generator of the process. The regime $p<p_c$ is called the {\sl ergodic regime} and we say that
an {\sl ergodicity breaking transition} occurs at the critical density
$p_c$.

Let
\[
g_p(\l):=p\sum_{i=k-j+1}^{k}\binom{k}{i}\l^i (1-\l)^{k-i}
\]
be the natural bootstrap percolation recursion map (cf.~\cite{MT})
associated to the OFA-jf process and let
\[
\tilde p :=\sup \{p\in[0,1]:~\l=0 \text{ is the unique fixed point of } g_p(\l)\}.
\]
In \cite{MT} it was proved that $p_c=\tilde p$ and that $p_c\in (0,1)$
for $j\ge 2$ and $p_c=1$ for $j=1$. Notice that, for $j=k$, the value
$\tilde p$ coincides with the site percolation threshold on $\bbT$ so
that $p_c=\tilde p=1/k$.

Consider now the finite volume OFA-jf process on $\O_{\bbT_L}$ and let
$\mu^t_\o$ be the law of the process at
time $t$ when the initial configuration is $\o$.
Further let $h^t_\o$ be the relative density of $\mu^t_\o$ w.r.t the
reversible stationary measure $\pi$.
Define the family of mixing times $\{T_a(L)\}_{a\ge 1}$ by
\[
T_a(L):= \inf\left\{t\ge 0:\ \max_\o\pi\left(|h^t_\o -1|^a\right)^{1/a}\le 1/4\right\}.
\]
Notice that $T_1(L)$ coincides with the usual mixing time $\tmix(L)$ of the chain  (see,
e.g.,~\cite{LPW}) and that, for any $a\ge 1$, one has $T_1(L)\le
T_a(L)$. Further let $\trel(L)$ be the relaxation time of the chain,
ie the inverse of the spectral gap of the generator $\cL_{\bbT_L}$.

\begin{theorem}[\cite{MT}]
\label{io e C}\
  \begin{enumerate}[(i)]
  \item Assume $p<p_c$ and consider  the finite
  volume OFA-jf model on $\O_{\bbT_L}$. Then
\[
\sup_L \trel(L)<\infty.
\]
If instead $p>p_c$ then $\trel(L)$ is exponentially large in $L$.
\item For all $p\in (0,1)$ there exists a constant $c>0$ such that
\[
T_2(L+1)-T_2(L)\le c \, \trel(L).
\]
In particular
\[
\tmix(L)\le T_2(L) \le c\, \trel(L)\, L
\]
  \end{enumerate}
\end{theorem}
The second result concerns the critical behavior $p=p_c$.
\begin{theorem}[\cite{CMRTtree}]
\label{noi}Consider the maximally constrained model $j=k$ and choose $p=p_c$. Then there exists $\b\ge 2$ and $c>0$ such that
\begin{align*}
  c^{-1} L^2\le \trel(L)\le c \,L^{\beta}.
\end{align*}
Moreover,
 \begin{equation*}
c^{-1}L T_{\rm rel}(L) \leq \tmix(L)\le T_2(L) \le
c L\,T_{\rm rel}(L).
 \end{equation*}

\end{theorem}

\subsection{Proof of Theorem~\ref{th:main3}}
We first need a preliminary result saying that, for infinitely many values of
$L$,
the increments of $T_{\rm hit}(L)$ can be controlled by the
corresponding relaxation time.
\begin{lemma}
\label{treelem:1} There exists a
constant $c_1$ such that, for all $\d>0$ and all $n$ large enough, the
following holds.
\begin{enumerate}[(a)]
\item In the maximally constrained model at $p\le p_c$
\[
\max\Bigl(T_{\rm hit}(L_n)-T_{\rm hit} (L_n-1),\ T_{\rm hit}(L_n+1)-T_{\rm hit} (L_n)\,\Bigr)\le \frac{c_1}{\d} \,\trel((1+\d)n),
\]
for some $L_n\in [n,(1+\d)n]$.
\item In the minimally constrained model
\[
T_{\rm hit}(L_n+1)-T_{\rm hit} (L_n)\ge -\frac{c_1}{\d} \,\trel((1+\d)n)
\]
for some $L_n\in [n,(1+\d)n]$.
\end{enumerate}
\end{lemma}
\begin{proof}
Fix $\d$ and $n\ge 1/\d$ and consider the maximally constrained model.
Using part (ii) of Theorem~\ref{io e C},
\begin{equation}
  \label{tree:1}
\tmix(n)\le T_2(n)\le c \sum_{i=1}^n \trel(i)\le c \, n\trel(n),
\end{equation}
where we used the fact that $\trel(i)\le \trel(n)$ for all $i\le n$.
Fix now $c_1>0$ and suppose that, for all $i\in [n, (1+\d)n-1]$,
\begin{gather*}
\max\Bigl(T_{\rm hit}(i+1)-T_{\rm hit}(i) ,\,T_{\rm hit}(i+1)-T_{\rm
  hit}(i)\,\Bigr)
\ge \frac{c_1}{\d}
\,\trel\left((1+\d)n\right).
\end{gather*}
In particular
\[
T_{\rm hit} ((1+\d)n)\ge c_1n \,\trel\left((1+\d)n\right) /2.
\]
On the other hand, using the results in\cite{Aldous},  there exists a
constant $\l=\l(p)$ such that
\begin{equation}
  \label{tree:2}
T_{\rm hit}((1+\d)n) \le \l \tmix((1+\d)n).
\end{equation}
In conclusion, using Theorem~\ref{io e C},
\begin{gather*}
\trel\left((1+\d)n\right) \le \frac{2}{c_1n }T_{\rm hit} ((1+\d)n)
\le \frac{2\l}{c_1n} \tmix((1+\d)n)\\\le
\frac{2\l c (1+\d)}{c_1}\trel((1+\d)n),
\end{gather*}
and we reach a contradiction by choosing
$c_1> 2\l c(1+\d)$.

Similarly, in the minimally constrained case, assume
\[
T_{\rm hit}(i+1)-T_{\rm hit}(i)\le -\frac{c_1}{\d}
\,\trel\left((1+\d)n\right),\quad \forall i\in [n, (1+\d)n-1],
\]
so that
\[
0\le T_{\rm hit} ((1+\d)n)\le T_{\rm hit}(L)-c_1n \,\trel\left((1+\d)n\right).
\]
Using again Theorem~\ref{io e C} together with \eqref{tree:2} we get
\begin{gather*}
\trel\left((1+\d)n\right)\le \frac{1}{c_1n  }T_{\rm hit}(L)\\
\le
\frac{\l}{c_1n } \tmix(L)\le
\frac{c\l}{c_1n } L
\trel(L)\le \frac{\l c (1+\d)}{c_1}\trel((1+\d)n).
\end{gather*}
and again we reach a contradiction by choosing
$c_1> \l c(1+\d)$.
\end{proof}
\subsubsection{Proof of theorem~\ref{th:main3} for the maximally
  constrained model}
The key observation here is that, for any $L\in \bbN$, the hitting
time $\t(L+1)$ is stochastically larger than the maximum between $k$
independent copies $\{\t^{(i)}(L)\}_{i=1}^k$ of the hitting time $\t(L)$. That follows
immediately by noting that:
\begin{itemize}
\item starting from the configuration identically equal to $1$, a
  vertex $x$ can be updated only after the first time at which all
  its $k$-children have been updated;
\item the projection of the OFA-jf process
on the sub-trees rooted at each one of the children of the root of $\bbT_{L+1}$
are independent OFA-jf processes on $\bbT_L$.
\end{itemize}
Henceforth, the proof follows from a beautiful argument of Dekking and Host that was used in~\cite{DH91}
to derive tightness for the minima of certain branching random walks.
\begin{align*}
  T_{\rm hit}(L+1)&\ge
  \bbE\bigl[\max_{i=1,\dots,k}\t^{(i)}(L)\bigr]\\
&\ge \frac 12 \bbE\bigl[\t^{(1)}(L)+\t^{(2)}(L) + |\t^{(1)}(L)-\t^{(2)}(L)|\bigr]\\
&=T_{\rm hit}(L) +\frac 12
\bbE\bigl[|\t^{(1)}(L)-\t^{(2)}(L)|\bigr]\\
&\ge T_{\rm hit}(L) +\frac 12 \bbE\bigl[|\bar \t^{(1)}(L)|\bigr],
\end{align*}
since whenever $X',X''$ are i.i.d.\ copies of a variable one has
$\E|X'-X''| %= \E[ \E [|X'-X''| \mid X''] ]
 \geq \E \left| X' - \E X'\right|$
by conditioning on $X''$ and then applying Cauchy-Schwarz. Altogether,
\begin{equation}
  \label{eq:tree7}
\bbE\bigl[|\bar \t^{(1)}(L)|\bigr] \le 2 \left(T_{\rm
    hit}(L+1)-T_{\rm hit}(L)\right).
\end{equation}
The conclusion of the theorem now follows from Lemma~\ref{treelem:1}
and Theorem~\ref{io e C}.
\qed

\subsubsection{Proof of theorem~\ref{th:main3} for the minimally constrained model}
In this case we define
\[
\t_{\rm min}(L):= \min_{i=1,\dots,k}
\tau^{(i)}(L),
\]
where $\tau^{(i)}(L)$ is the first time that the
$i^{th}$-child of the root of $\bbT_{L+1}$ is updated and we write
\[
T_{\rm hit}(L+1)\le  \bbE\bigl[\t_{\rm min}(L)\bigr] + \sup_L\sup_{\o\in \cG_L} \bbE_\o\bigl[\t(L)\bigr],
\]
with $\cG_L$ the set of configurations in $\O_{\bbT_{L}}$ with
$\o_r=1$ and at
least one zero among the children of the children of the root $r$.
\begin{lemma}
\label{l.1}
$
\sup_L\sup_{\o\in \cG_L} \bbE_\o \t(L)< \infty.
$
\end{lemma}
Assuming the lemma we write
\begin{align*}
  T_{\rm hit}(L+1)&\le \bbE\bigl[\t_{\rm min}(L)\bigr] +c\\
&\le \frac 12 \bbE\bigl[\t^{(1)}(L)+\t^{(2)}(L) -
  |\t^{(1)}(L)-\t^{(2)}(L)|\bigr] +c\\
&=T_{\rm hit}(L) -\frac 12
\bbE\bigl[|\t^{(1)}(L)-\t^{(2)}(L)|\bigr]+c.
\end{align*}
Thus
\[
\bbE\bigl[|\bar \t(L)|\bigr]\le \bbE\bigl[|\t^{(1)}(L)-\t^{(2)}(L)|\bigr]
\le 2\bigl(T_{\rm hit}(L)-T_{\rm hit}(L+1)\bigr) +2c.
\]
Hence, if $L_n\in [n, (1+\d)n]$ satisfies property (b) of Lemma~\ref{treelem:1}, we get
\[
\bbE\bigl[|\bar \t^{(1)}(L)|\bigr]\le 2 \frac{c_1}{\d}
\trel\left((1+\d)n\right) +2c.
\]
The conclusion of the theorem now follows from Theorem~\ref{io e C}.
\qed
\begin{proof}[Proof of Lemma~\ref{l.1}]
Fix $L$
and $\o\in \cG_L$ and observe that
\begin{gather*}
\bbP_\o(\o_r(t)=1)\\
=\bbP_\o(\o_r(t)=1\mid \t(L)\ge t)\bbP_\o(\t(L)\ge t)+
\bbP_\o(\o_r(t)=1\mid \t(L)<t)\bbP_\o(\t(L)<t) \\
= (1-p)\bbP_\o(\t(L)\ge t) +p.
\end{gather*}
That is because $\o_r=1$ at time $t=0$ while it is a Bernoulli(p)
random variable given that the root has
been updated at least once. Thus
\[
 \bbE_\o[\t(L)]\le \frac{1}{1-p}\int_0^\infty dt  \,
 |\bbP_\o(\o_r(t)=1) -p |.
\]
In order to bound from above the above integral we closely follow the strategy of \cite{CMST}*{\S4}. In what
follows, for any finite subtree $\cT$ of $\bbT$, we will refer to the \emph{children} of $\cT$ as the vertices of $\bbT\setminus \cT$
with their parent in $\cT$. Using the graphical construction, for all times $t\ge 0$ we define
%a (random) \emph{distinguished} set of zeros $\cZ_t$ and
a (random) \emph{distinguished} tree $\cT_t$ according to the following algorithm:
\begin{enumerate}[(i)]
\item
$\cT_0$ coincides with the root together with those among its children which have
at least one zero among their children (\ie they are unconstrained).
\item %$\cZ_t=\cZ_0$ and
$\cT_t=\cT_0$ until the first ``legal'' ring at time $t_1$ at one of
the children of $\cT_0$, call it $x_0$.
\item  $\cT_{t_1}=\cT_0\cup \{x_0\}$.
\item Iterate.
\end{enumerate}
Exactly as in \cite{CMST}*{\S4.1}, one can easily verify the
following key properties of the above construction:
\begin{enumerate}[(a)]
\item for all $t\ge 0$ each leaf of $\cT_t$ is unconstrained \ie
  there is a zero among its children;
\item if at time $t=0$ the variables $\{\o_x\}_{x\in \cT_0}$ are not
  fixed by instead are i.i.d with law $\pi$, then, conditionally on $\{\cT_s\}_{s\le t}$,
  the same is true for the variables $\{\o_x(t)\}_{x\in \cT_t}$.
\item For all $i\ge 1$, given $\cT_{t_i}$ and $t_i$, the law of the
  random time $t_{i+1}-t_i$ does not depend on the variables (clock
  rings and coin tosses) of the graphical construction in $\cT_{t_i}$.
\end{enumerate}
As in \cite{CMST}*{Eqs.~(4.8) and~(4.10)}, the above properties imply
that
\begin{align*}
\var_\pi(\bbE_\o\left[\o_r(t)\mid \{\cT_s\}_{s\le t}\right]\le
e^{-2t/\trel(L)}.
%\max_{\o\in \cG_L}|\bbE_\o\left(\o_r(t)\right)|&\le \frac{1}{p\wedge q}e^{-t/\trel(L)}.
\end{align*}
Therefore,
\begin{align*}
    \sup_{\o\in \cG_L}\big|\,\bbE_\o\left[\o_r(t)-p\right]\big|&\le
\sup_{\o\in \cG_L}\bbE_\o\,\Big|\, \bbE_\o\left[\o_r(t)-p\mid  \{\cT_s\}_{s\le
    t}\right]\Big|  \\
&\le \left(\frac{1}{p\wedge q}\right)^{|\cT_0|}
\sup_{\o\in \cG_L}\bbE_\o\Bigl[\sum_{\o\in \O_{\cT_0}}\pi(\o)\big|\,\bbE_\o\left(\o_r(t)-p\mid \{\cT_s\}_{s\le t}\right)\big|\Bigr]\\
&\le \left(\frac{1}{p\wedge q}\right)^{|\cT_0|}\sup_{\o\in \cG_L}\bbE_\o\left[\var_{\pi}\left(\bbE_\o\left(\o_r(t)\tc \{\xi_s\}_{s\le t}\right)\right)^{1/2}\right]
\\
&\le \left(\frac{1}{p\wedge q}\right)^{|\cT_0|}e^{-t/\trel(L)}\,.
\end{align*}
By Theorem~\ref{io e C} we have that  $\sup_L\trel(L)<\infty$, and the proof is complete.
\end{proof}
Consider the maximally constrained process on $\O_{\bbT_{L+1}}$ and
let $\t^{\rm max}(L)$ be the first time at which all the children of
the root have been updated at least once starting from the configuration identically
equal to one. For a given $\o\in \O_{\bbT_{L+1}}$ and $x\in
\bbT_{L+1}$, further let $\cC_\o(x)$ be the maximal subtree rooted at $x$ where
$\o$ is equal to one.
Finally, recall that
$\bbP(\cdot)$ denotes the basic coupling given by the graphical
construction and that $\o(t)$ denotes the process at time $t$ started
from the initial configuration $\o$.
\begin{lemma}
\label{l.1bis} There exists some $c>0$ such that
\begin{align*}
\max_{\o\in \O_{\bbT_{L+1}}}\bbP\left( |\cC_{\o(\t^{\rm
      max}(L))}(r)|\ge n\right)&\le c\, \pi\left(|\cC_\o(r)|\ge \frac{n-2}{k-1}\right),
\end{align*}
and in particular,
\[
\max_{\o\in \O_{\bbT_{L+1}}}\bbE\left|\cC_{\o(\t^{\rm max}(L))}(r)\right|\le c \sum_\o\pi(\o)|\cC_\o(r)|.
\]
\end{lemma}
\begin{proof}
Recall that under the basic coupling all the
starting configurations have coupled by time $\t^{\rm max}(L)$. Hence,
\begin{align*}
\bbP&\left(\exists \,\o \in\O_{\bbT_{L+1}}:\ |\cC_{\o(\t^{\rm max}(L))}(r)|\ge n\right) =
\sum_\o\pi(\o) \bbP\left(|\cC_{\o(\t^{\rm max}(L))}(r)|\ge n\right)\\
&\le k \sum_\o\pi(\o) \bbP\left(|\cC_{\o(\t^{(1)}(L))}(r)|\ge n\, ,\, \t^{\rm max}(L)=\t^{(1)}(L)\right) \\
&\le k \sum_\o\pi(\o) \bbP\left(|\cC_{\o(\t^{(1)}(L))}(r)|\ge n\right) ,
\end{align*}
where $\t^{(1)}(L)$ is the first time that the first (in some chosen order)
child of the root has been updated starting from all ones. By
construction, at time $\t^{(1)}(L)$ the first child has all its
children equal to zero. Therefore the event $\{\cC_{\o(\t^{(1)}(L))}(r)|\ge n\}$ implies that there exists some other child $x$ of the
  root such that $\cC_{\o(\t^{(1)}(L))}(x)$ has cardinality at least $(n-2)/(k-1)$. Using reversibility and the independence between
  $\t^{(1)}(L) $ and the
  process in the subtree of depth $L$ rooted at $x$ together with a
  union bound over the choice of $x$, we conclude that
  \begin{gather*}
    \sum_\o\pi(\o) \bbP\left(|\cC_{\o(\t^{(1)}(L))}(r)|\ge n \right)\le
(k-1)\pi\left(|\cC_\o(r)|\ge \frac{n-2}{k-1}\right).
  \end{gather*}
The statement of the lemma follows at once by summing over $n$.
\end{proof}
Using the lemma we can now prove the analogue of Lemma~\ref{l.1}
\begin{lemma}
\label{l.2} Fix any positive integer $\ell$. For all $p\le p_c$ there exists $c=c(\ell,p)$ such that
\begin{align*}
&(i) &T_{\rm hit}(L+\ell)&\le \bbE\left[\t^{\rm max}(L)\right]+ c\,\trel(L)\qquad
&\text{if } p<p_c,\\
&(ii)&T_{\rm hit}(L+\ell)&\le \bbE\left[\t^{\rm max}(L)\right]+ c L \trel(L) \qquad
&\text{if } p=p_c.
\end{align*}
Moreover,  for any $d>0$,
\begin{equation}
\label{eq:tail}
  \bbP \bigl(\t(L+\ell)-\t^{\rm max}( L) \ge d\,T_{\rm
  rel}\bigr)=
\begin{cases}
 O(d^{-1})&\text{ if }p<p_c,\\
O(d^{-1/3})&\text{ if }p=p_c.\end{cases}
\end{equation}
\end{lemma}
\begin{proof}
For simplicity we give a proof for the case $\ell=1$. The general proof is similar and we omit the details. We first claim that, starting from $\o\in \O_{\bbT_{L+1}}$, one has
\begin{equation}
  \label{eq:tree15}
\bbE_\o[\t(L+1)]\le c\, |\cC_\o|\trel(L)
\end{equation}
for some constant $c$, where $|\cC_\o|$ denotes the cardinality of $\cC_\o$.
If we assume the claim, the strong Markov property implies that
\[
T_{\rm hit}(L+1)\le \bbE\left[\t^{\rm max}(L)\right]+c\, \bbE\left[|\cC_{\o(\t^{\rm max}(L))}|\right]\,\trel(L)
\]
where all expectations are computed starting
from all ones. Using Lemma~\ref{l.1bis},
\[
\bbE\left[|\cC_{\o(\t^{\rm
      max}(L))}|\right]\le c' \sum_\o\pi(\o)|\cC_\o(r)|
\] for some constant $c'$ and parts (i) and (ii) of the lemma
follow by standard results on percolation on regular trees (see, e.g.,~\cite{Grimmett}).

To prove \eqref{eq:tree15} we proceed exactly as in Lemma
\ref{l.1}. We first write
\[
 \bbE_\o\left[\t(L+1)\right]\le \frac{1}{1-p}\int_0^\infty dt  \,
 |\bbP_\o(\o_r(t)=1) -p |
\]
and then we apply the results of \cite{CMST}*{\S4} to get that
\[
|\bbP_\o(\o_r(t)=1) -p |\le \min\left[1,\left(\frac{1}{p\wedge q}\right)^{|\cC_\o|}e^{-t/\trel(L)}\right].
\]
Thus,
\[
\frac{1}{1-p}\int_0^\infty dt  \, |\bbP_\o(\o_r(t)=1) -p |\le c
|\cC_\o|\, \trel(L)
\]
for some constant $c$ and \eqref{eq:tree15} follows.

Lastly we prove \eqref{eq:tail}. The subcritical case $p<p_c$ follows easily from $(i)$ and Markov's inequality, while the critical case follows from \eqref{eq:tree15}. To see this, write
\begin{align*}
\bbP&\bigl(\t(L+1)-\t^{\rm max}(L)\ge d\,\trel(L)\bigr)\\
&=\bbP\bigl(\t(L+1)-\t^{\rm max}(L)\ge d\,\trel(L)\, ,\, |\cC_{\o(\t^{\rm max}(L))}| \le {d}^{2/3}\bigr)\\
&+\bbP\bigl(\t(L+1)-\t^{\rm max}(L)\ge d\,\trel(L)\,,\,|\cC_{\o(\t^{\rm max}(L))}| >{d}^{2/3}\bigr).
\end{align*}
Using Markov's inequality and \eqref{eq:tree15},
\begin{align*}
\bbP&\bigl(\t(L+1)-\t^{\rm max}(L)\ge d\,\trel(L)\, ,\,
|\cC_{\o(\t^{\rm max}(L))}| \le {d}^{2/3}\bigr)\\
&\le \frac{1}{d \trel(L)}\,
\bbE\left[{\mathds 1}_{\{|\cC_{\o(\t^{\rm max}(L))}| \le
    d^{2/3}\}}\bbE_{\o(\t^{\rm
      max}(L))}\left[\t(L+1)\right]\right]\\
&\le \frac{c}d \, \bbE\left[{\mathds 1}_{\{|\cC_{\o(\t^{\rm max}(L))}| \le
    d^{2/3}\}}|\cC_{\o(\t^{\rm max}(L))}|\right] \leq c {d}^{-1/3}.
\end{align*}
The second term is also $O\left(d^{-1/3}\right)$
using Lemma~\ref{l.1bis} and the fact that, for $p=p_c$,
\begin{equation*}
\pi\left(|\cC_\o|\ge
   n\right)= O(1/\sqrt{n}).
\qedhere
\end{equation*}
\end{proof}

\subsection{Proof of Theorem~\ref{th:main4}}
Fix $\e\in (0,1/2)$.  Let  $\{L_n\}$ be a sequence such that,
for all $n$ large enough,
\begin{equation}
  \label{eq:tree12}
\max\Bigl(T_{\rm hit}(L_n)-T_{\rm hit}(L_n-1) ,T_{\rm hit}(L_n+1)-T_{\rm hit}(L_n)\Bigr)\le c \trel(L_n),
\end{equation}
for some constant
$c$ independent of $n$. The existence of such a sequence is guaranteed
by Lemma~\ref{treelem:1}. We begin by proving that
\begin{equation}
  \label{eq:tree11}
\tmix(L_n,\e)\le T_{\rm hit}(L_n)+ O_\e(\trel(L_n)).
\end{equation}
Exactly as for the East process, one readily infers from the
graphical construction that at time $\t(L_n)$ all initial configurations
$\o\in \O_{\bbT_{L_n}}$ have coupled. Therefore (cf.~\S\ref{East-cutoff}),
\[
\max_{\o,\o'}\left\|\bbP^{\bbT_{L_n},t}_\o-\bbP^{\bbT_{L_n},t}_{\o'}\right\|\le
\bbP(\t(L_n)>t).
\]
If $t=T_{\rm hit}(L_n)+\D $, Markov's inequality together with
\eqref{eq:tree7}
imply that
\begin{align*}
\bbP\left(\t(L_n)>T_{\rm hit}(L_n)+\D\right)&\le
\frac{1}{\D}\bbE\left(|\bar \t(L_n)|\right)
\le \frac{2}{\D}  \bigl[T_{\rm hit}(L_n+1)-T_{\rm
  hit}(L_n)\bigr]\\
&\le \frac{2}{\D} c\,\trel(L_n).
\end{align*}
Inequality \eqref{eq:tree11} now follows by choosing
$\D=2c\,\trel(L_n)/\e$.

Next we prove the lower bound
\begin{equation}
  \label{eq:tree11bis}
\tmix(L_n,1-\e)\ge T_{\rm hit}(L_n)- O_\e(\trel(L_n)).
 \end{equation}
Start the process from the configuration $\o$ identically equal to one
and let $\t^{\rm max}(L_n-\ell)$ be the time when all the vertices at distance $\ell$ from the root have been updated at least once.
Conditionally on $\t^{\rm max}(L-\ell)>t$,  the root is connected by a path of $1's$ to some vertex at distance $\ell$
at time $t$. On the other hand, standard percolation results for $p\le
p_c$ imply that the $\pi$-probability
of the above event is smaller than $\e/2$ provided that $\ell$ is
chosen large enough. Therefore, for such value of $\ell$,
\begin{align*}
\|\mu^t_\o-\pi\|&\ge \bbP(\t^{\rm max}(L_n-\ell) >t)-\e/2.
\end{align*}
It remains to show that
\[
\bbP(\t^{\rm max}(L_n-\ell) >t)\ge 1-\frac{\e}{2},
\]
for $t= T_{\rm  hit}(L_n)-O_\e(\trel(L))$.

We prove this by
contradiction. Let $t=T_{\rm
  hit}(L_n)-DT_{\rm rel}$, where $D$ is a constant to be specified
later, and suppose that $\bbP(\t^{\rm max}(L_n-\ell) >t)<
1-\frac{\e}{2}$. Using Lemma~\ref{l.2} we can choose a large constant $\D$ independent of $L_n$ such that
\[
\bbP(\t(L_n)-  \t^{\rm max}(L_n-\ell) \ge \D T_{\rm rel})\le \e/4,
\] and hence, by a union bound,
$$\bbP\bigl(\t(L_n)<t+\D T_{\rm rel}\bigr)>\e/4.$$
However, for large enough $D$, this contradicts Theorem~\ref{th:main3}.
Theorem~\ref{th:main4} now follows from \eqref{eq:tree11},
\eqref{eq:tree11bis}, Theorems~\ref{io e C} and~\ref{noi}, and Lemma~\ref{treelem:1}. \qed

\subsection*{Acknowledgments}
We are grateful to Y. Peres for pointing out the relevant literature on branching random walks, which led to improved estimates in Theorems~\ref{th:main3}--\ref{th:main4}. We also thank O. Zeitouni for an interesting conversation about the concentration results on trees and O. Blondel for several useful comments.
This work was carried out while F.M.\ was a
  Visiting Researcher at the Theory Group of Microsoft Research
 and S.G.\ was an intern there; they thank the group for its hospitality.

\end{document}